\documentclass[12pt,a4paper]{article}

\usepackage[english]{babel}

\usepackage[letterpaper,top=2cm,bottom=2cm,left=3cm,right=3cm,marginparwidth=1.75cm]{geometry}
\usepackage{tikz}
\usepackage{amsmath}
\usepackage{graphicx}
\usepackage[colorlinks=true, allcolors=blue]{hyperref}
\usepackage{amsmath, amssymb, amsthm, amsfonts}
\usepackage{fullpage}
\usepackage{enumitem}
\usepackage{mathtools}
\usepackage{caption}
\usepackage[capitalise]{cleveref}
\usepackage{amsmath,amssymb,amsthm,tikz,xspace,subfigure}
\usetikzlibrary{positioning, shapes.geometric, calc, intersections}
\tikzstyle{vertex}=[circle,fill=black,inner sep=2pt]
\tikzstyle{vertrect}=[draw,rectangle,inner sep=2pt]
\tikzstyle{vertdia}=[draw,diamond,inner sep=2pt]

\usepackage{etoolbox}

\usepackage{amssymb,comment,amsmath,amsthm}

\usepackage{xcolor}
\usepackage{titling,titlesec}
\usepackage{url}
\usepackage{lipsum}
\newfont{\blb}{msbm10 scaled\magstep1}
\newfont{\comp}{cmr12 scaled\magstep1}
\newfont{\compb}{cmr10 scaled\magstep2}
\newfont{\sbb}{cmssbx10 scaled\magstep3}
\newfont{\sbbb}{cmssbx10 scaled\magstep5}
\newfont{\sbs}{cmssbx10 scaled\magstep1}

\theoremstyle{plain}
      \newtheorem{theorem}{Theorem}[section]
      \newtheorem{lemma}[theorem]{Lemma}
            
      \newtheorem{problem}[theorem]{Problem}
      
      \newtheorem{corollary}[theorem]{Corollary}
      \newtheorem{proposition}[theorem]{Proposition}
      \newtheorem{conjecture}[theorem]{Conjecture}
\theoremstyle{definition}

\newcommand{\dm}[1]{\textcolor{blue}{\textbf{[DM:} #1\textbf{]}}}

\newcommand{\js}[1]{\textcolor{red}{\textbf{[JS:} #1\textbf{]}}}

\def\ex{\mathrm{ex}}

\newcommand{\RR}{{\mathbb R}}
\newcommand{\NN}{{\mathbb N}}
\newcommand{\Zd}{{\mathbb Z}^d}
\newcommand{\Ed}{{\mathrm{E}}^d}
\newcommand{\PP}{{\mathbb P}}
\newcommand{\EE}{{\mathbb{E}}}
\newcommand{\al}{\alpha}
\newcommand{\be}{\beta}
\newcommand{\ga}{\gamma}
\newcommand{\de}{\delta}

\title{Many pentagons in triple systems}

\author{
Dhruv Mubayi\thanks{Department of Mathematics, Statistics and Computer Science, University of Illinois, Chicago, IL 60607. Email: mubayi@uic.edu. Research partially supported by NSF Award DMS-2153576.} \and
Jozsef Solymosi\thanks{University of British Columbia, Vancouver, Canada, and Obuda University, Budapest, Hungary
			Email: solymosi@math.ubc.ca
			Research supported by an NSERC Discovery grant and OTKA K grant no. 133819.}}

\begin{document}

\date{ }

\maketitle

\begin{abstract}
 We prove that every $n$ vertex linear triple system with $m$ edges has at least $m^6/n^7$
 copies of a  pentagon, provided $m>100 \, n^{3/2}$. This provides the first nontrivial bound for a question posed by Jiang and Yepremyan. 
 
 More generally, for each $ \ell \ge 2$, we  prove that there is a constant $c$ such that if an $n$-vertex graph is $\varepsilon$-far from being triangle-free, with $\varepsilon \gg n^{-1/3\ell}$, then it has at least $c \, \varepsilon^{3\ell} n^{2\ell+1}$ copies of $C_{2\ell+1}$. This improves the previous best bound of $c \, \varepsilon^{4\ell+2} n^{2\ell+1}$ due to Gishboliner, Shapira and Wigderson.  
 
 Our result also yields some geometric theorems, including the following. For $n$ large, every $n$-point set in the plane with at least $60\, n^{11/6}$ triangles similar to a given triangle $T$, contains two triangles sharing a special point, called the harmonic point. In the other direction, we give a construction showing that the exponent $11/6\approx 1.83$ cannot be reduced to anything smaller than $\approx 1.726$.

\end{abstract}

\section{Introduction}

We consider the supersaturation problem for odd cycles in linear 3-graphs (triple systems) and show some applications of this question. A (loose or linear) cycle $C_{k}$ is the 3-graph containing $k$ distinct vertices
$v_1, \ldots, v_k$ and $k$ distinct edges $e_1, \ldots, e_k$ such that $e_i$ is obtained by enlarging $\{v_i, v_{i+1}\}$ by a new vertex $w_i$ such that 
$w_1, \ldots, w_k$ are all distinct and distinct from all the $v_j$s.  In other words (taking indices modulo $k$), 
$$V(C_k) = \{v_1, \ldots, v_k, w_1, \ldots, w_k\} \qquad \hbox{ and } \qquad 
E(C_k)= \{v_iv_{i+1}w_i: i =1 \ldots, k\}.$$

A triple system is {\em linear} if every pair of vertices lies in at most one edge.
A natural extremal problem is to determine the Tur\'an number $\ex_L(n, C_k)$, defined as the maximum number of edges in an $n$-vertex linear triple system that contains no copy of $C_k$ as a (not necessarily induced) subgraph. We are especially interested in the case when $k=2\ell+1$ is odd.
The case of $C_3$ is special as determining $\ex_L(n, C_3)$ is, apart from constant multiplicative factors,  equivalent to the well-known $(6,3)$-problem of Brown-Erd\H os-S\'os. Here, famous results of Behrend~\cite{B} and Ruzsa-Szemer\'edi~\cite{RSz} show that $\ex_L(n, C_3)=n^{2-o(1)}$ with the $o(1)$ term being a function of intense study over the years.

We are mainly concerned with the case of $C_5$, henceforth called the pentagon.  The first author, Kostochka, and Verstra\"ete proved that $\ex_L(n, C_5) = \Omega(n^{3/2})$ while writing the paper~\cite{KMV} in 2013.  Theorem 1.2
in Collier-Cartaino, Graber,  Jiang~\cite{CGJ}  refers to this result. This was later published in~\cite{EGM} by Ergemlidze, Gy\H ori, Methuku. More generally, the upper bound $\ex_L(n, C_{2\ell+1}) = O(n^{1+1/\ell})$ was proved in~\cite{CGJ}; no corresponding lower bound is known for any $\ell>2$. 

 Many extremal problems exhibit the property that when the underlying (typically large) discrete object is dense enough to contain a given forbidden subobject, it contains many of them.  In our context, this means that $n$-vertex triple systems with $m$ edges contain many pentagons when  $m$ is much larger than $n^{3/2}$. Indeed, our main result quantifies this dependence.

\begin{theorem} \label{thm:mainC5}
Let $n>10$ and let $H$ be an $n$-vertex linear triple system with $m > 100 \, n^{3/2}$ edges. Then the number of copies of $C_5$ in $H$ is at least $m^6/n^7$.
\end{theorem}

Sidorenko's conjecture states that the homomorphism density of a graph $G$ in a graph $W$ is at least the edge density of $W$ raised to the power $e$, where $e$ is the number of edges in $G$. This is known to be false for some hypergraphs, but deciding if it is true for the pentagon when the underlying triple system is linear seems interesting. Namely, can the quantity $m^6/n^7$ in Theorem~\ref{thm:mainC5} be improved to $m^5/n^5$, which, if true, would be sharp in order of magnitude as shown by random triple systems? This problem was posed by Jiang and Yepremyan~\cite{JY}. We believed that no such improvement is possible, and in fact, after our preprint was made public, this was shown to be true  by Methuku (personal communication), who gave a construction where the number of $C_5$ is $O((m^5/n^5)^{1-\varepsilon})$. Further, we conjecture that the truth is $\Theta(m^6/n^7)$, but this remains open. The following result provides some motivation for our conjecture. Throughout this paper, we use standard asymptotic notation.
  \begin{proposition} \label{thm:evidence}
        Suppose that for all $n$ there exists  a linear triple system $H$ on $n$ vertices and  $m=\Theta(n^{3/2})$ edges, with maximum degree $O(n^{1/2})$, at most $O(n^{3/2})$ copies of $C_3$, and at most $O(n^2)$ copies of $C_k$ for $k=4,5$. Then, the bound $m^6/n^7$ in Theorem~\ref{thm:mainC5} is tight in order of magnitude.
    \end{proposition}
We provide a construction of a linear $H$ as in Proposition~\ref{thm:evidence} with {\em no} copies of $C_3$ and $C_5$, but the number of copies of $C_4$ is $\Omega(n^{5/2})$. We remark that a linear 3-graph $H$ as in Proposition~\ref{thm:evidence} with {\em no}  of copies of $C_k$ for each $3\le k \le 5$ does not exist, as it was shown by Conlon, Fox, Sudakov, and Zhao~\cite{CFSZ} that any such $H$ has $o(n^{3/2})$ edges.  
    
A natural generalization of Theorem~\ref{thm:mainC5}  has an application to a problem concerning quantitative aspects of removal lemmas in graphs, which are, in turn, connected to questions about property testing. Say that an $n$-vertex graph $G$ is $\varepsilon$-far from being triangle-free if the minimum number of edges required to be removed from $G$ to make $G$ triangle-free is at least $\varepsilon n^2$. Gishboliner, Shapira and Wigderson~\cite{GSW} proved that there is a constant $c$ such that if $\varepsilon>0$ and $n$ is sufficiently large in terms of $\varepsilon$, and $G$ is $\varepsilon$-far from being triangle-free, then  $G$ has at least $c \, \varepsilon^{4\ell+2} n^{2\ell+1}$ copies of $C_{2\ell+1}$. We improve this as follows. We write $a \gg b$ to denote that there is a sufficiently large constant $C>0$ such that 
$a > C\,b$.
\begin{theorem} \label{thm:epsfar}
  Fix $\ell \ge 2$. There is a constant $c$ such that if an $n$-vertex graph $G$ is $\varepsilon$-far from being triangle-free, with $\varepsilon \gg n^{-1/3\ell}$, then $G$ has at least $c \, \varepsilon^{3\ell} n^{2\ell+1}$ copies of $C_{2\ell+1}$. 
  \end{theorem}
  
  The lower bound requirement on $\varepsilon$ in Theorem~\ref{thm:epsfar} can be weakened, but we make no attempt to optimize its value (the optimal value would be $n^{-1+1/\ell}$).  We note that the exponent $3\ell$ of $\varepsilon$ cannot be improved to anything smaller than $2\ell+1$ as shown by random graphs.

 This paper summarizes unpublished works by the authors previously presented in seminars and conferences like in \cite{Conf1,Conf2}. The following related results were published independently by others: a  proof of a slight weakening of Theorem \ref{thm:mainC5}, and Theorem~\ref{thm:epsfar} for $\ell=2$ was published recently in \cite{GHIM}, 
 and, as mentioned earlier, a construction similar to the one in Section~\ref{C5Free}, was published in \cite{EGM}.

\subsection{An application in geometry}

Theorem~\ref{thm:mainC5}  provides a somewhat unexpected application to a problem in discrete geometry that we now describe.

Elekes and Erd\H{o}s proved in \cite{EE} that for any triangle $T$, there are $n$-element planar point sets $S$ with $\Omega(n^2)$ triangles similar to $T$.
It was proved shortly after that if the number of equilateral triangles in $S$ is least ($1/6+\varepsilon)n^2$, then $S$  contains large parts of a triangular lattice.
On the other hand, no lattice is guaranteed if $S$ contains at most $c\, n^2$ similar copies for $c<1/6$. Weaker, local structural properties of point sets with a quadratic number of similar triangles were proved in \cite{AFE}. We will prove that point sets with sub-quadratic (but still many) triangles, similar to a given $T$, are guaranteed to contain certain interesting local substructures.

We use complex numbers to represent points of the plane. A point $P$ with coordinates $(a,b)$ is represented by the 
complex number $z_P=a+ib$.  A {\em cyclic quadrilateral} $ABCD$ is a quadrilateral that can be inscribed in a circle. A {\em harmonic quadrilateral} is a cyclic quadrilateral in which the product of one pair of opposite sides is equal to the product of the other pair of opposite sides \cite{J,G}. This property can also be described using complex numbers and the harmonic cross-ratio:
\begin{equation} \label{cr=-1}(z_A, z_B; z_C, z_D) = -1, \end{equation}
where the cross-ratio is defined as:
\[ (z_A, z_B; z_C, z_D) = \frac{(z_A - z_C)(z_B - z_D)}{(z_A - z_D)(z_B - z_C)}. \]

The cross-ratio is important in analyzing point sets with a quadratic number of triangles and quadrangles in the plane. Laczkovich and Ruzsa proved that for a quadrilateral, $Q,$ there exist arbitrarily large point sets with a quadratic number of quadrangles similar to $Q$ if and only if the cross-ratio of $Q$ is algebraic~\cite{LR}. We will refer to this result as the Laczkovich-Ruzsa Theorem.

A simple calculation shows that given three points $z_A$, $z_B$, and $z_C$ in the complex plane, the fourth point $z_D$ such that the quadrilateral is harmonic can be expressed as
\begin{equation}\label{quad}
 z_D = \frac{2z_Az_B - z_Az_C-z_Bz_C}{z_A + z_B - 2z_C}.  
\end{equation}  

A point $D$ is a \textit{harmonic point} of triangle $ABC$ if $ABCD$ forms a harmonic quadrangle. 
By continuity, a triangle has three (distinct) harmonic points on its circumcircle, one in each of the three sectors of the circle between vertices of the triangle   (See Figure \ref{fig:harmonic} for some examples).
Moreover,  $D$ is the harmonic point of $ABC$ on the opposite side of $AB$ as $C$ iff $z_D$ satisfies (\ref{cr=-1}) or  (\ref{quad}).

\begin{figure}[h]
\centering
 \captionsetup{justification=centering}
\begin{tikzpicture}
    % Define radius for the hexagon
    \def\hexRadius{2.8cm}

   % Draw the circumcircle around the hexagon
    \draw[blue, thick] (0,0) circle(\hexRadius);  % Circumcircle with the given radius

    % Shade triangle BFD
     \fill[gray!20]  (210:\hexRadius)-- (90:\hexRadius)  -- (330:\hexRadius)  -- cycle;
    \draw           (210:\hexRadius) -- (90:\hexRadius) -- (330:\hexRadius)-- cycle;

    % Draw the triangle BFD
    \draw[thick, red] (210:\hexRadius) -- (90:\hexRadius) -- (330:\hexRadius) -- cycle;

    % Draw the hexagon and place points on the corners
    \foreach \i in {150,90,30,330,270,210} {
        \draw[thick] (\i:\hexRadius) -- ({\i+60}:\hexRadius);
        \fill ({\i:\hexRadius}) circle (2pt);  % Add a point on the corner of the hexagon
    }
    
    % Add labels outside the hexagon in math font (counterclockwise)
    \node at (150:3.1cm) {\( A \)};
    \node at (210:3.1cm) {\( B \)};
    \node at (270:3.1cm) {\( C \)};
    \node at (330:3.1cm) {\( D \)};
    \node at (30:3.1cm) {\( E \)};
    \node at (90:3.1cm) {\( F \)};

    % Draw a square to the right of the hexagon
    \begin{scope}[xshift=7.5cm]
        % Define the side length of the square
        \def\squareSide{\hexRadius * sqrt(2)}  % Side length based on the circumradius
        
        % Define the circumradius of the square (distance from center to corner)
        \def\squareCircumRadius{\squareSide / sqrt(2)}

        % Draw the square 
        \draw[thick] 
            ({-0.5*\squareSide},{0.5*\squareSide}) --   % top left
            ({0.5*\squareSide},{0.5*\squareSide}) --    % top right
            ({0.5*\squareSide},{-0.5*\squareSide}) --   % bottom right
            ({-0.5*\squareSide},{-0.5*\squareSide}) --  % bottom left
            cycle;

   % Add labels outside the square (even further away from the corners)
        \node at ({-0.5*\squareSide-7},{0.5*\squareSide+6}) {\( G \)};
        \node at ({0.5*\squareSide+7},{0.5*\squareSide+6}) {\( K \)};
        \node at ({0.5*\squareSide+6},{-0.5*\squareSide-4}) {\( I \)};
        \node at ({-0.5*\squareSide-8},{-0.5*\squareSide-4}) {\( H \)};

    % Points G and I coordinates
    \coordinate (G) at ({-0.5*\squareSide},{0.5*\squareSide});
    \coordinate (I) at ({0.5*\squareSide},{-0.5*\squareSide});

    % Angle for point L and calculate coordinates
    \def\angleL{80}  % Angle for point L
    \coordinate (L) at ({\squareCircumRadius * cos(\angleL)}, {\squareCircumRadius * sin(\angleL)});
    
    % Draw the line from L to G
    \draw[thick] (L) -- (G);

   % Coordinates for point J (top-right corner of the square)
    \coordinate (K) at ({0.5*\squareSide}, {0.5*\squareSide});

    % Draw the line from L to J
    \draw[thick] (L) -- (K);

        % Draw the circumcircle of the square (touching the corners)
        \draw[thick, blue] (0,0) circle [radius=\squareCircumRadius];

 % Shade triangle BFD
     \fill[gray!20]  (G)-- (K)  -- (I)  -- cycle;
    \draw           (G) -- (K) -- (I)-- cycle;

    % Draw the line GI, GJ, JI
    \draw[thick, red] (G) -- (I);
 \draw[thick, red] (G) -- (K);
     \draw[thick, red] (I) -- (K);
     
        % Move point L to the 80-degree position on the circumcircle
        \def\angleL{80}  % Angle for point L
        \fill[black] 
            ({\squareCircumRadius * cos(\angleL)}, {\squareCircumRadius * sin(\angleL)}) circle (2pt);  % Point L at 80 degrees
        \node at 
            ({\squareCircumRadius * cos(\angleL)}, {\squareCircumRadius * sin(\angleL) + 7.4}) {\( L \)};  % Label L slightly above the point

        % Move point M to the -10-degree position on the circumcircle
        \def\angleM{-10}  % Angle for point M
        \fill[black] 
            ({\squareCircumRadius * cos(\angleM)}, {\squareCircumRadius * sin(\angleM)}) circle (2pt);  % Point M at -10 degrees
        \node at 
            ({\squareCircumRadius * cos(\angleM) + 7.3}, {\squareCircumRadius * sin(\angleM) + 0.4}) {\( J \)};  % Label M slightly above the point

    % Angle for point M and calculate coordinates
    \def\angleL{-10}  % Angle for point L
    \coordinate (J) at ({\squareCircumRadius * cos(\angleL)}, {\squareCircumRadius * sin(\angleL)});
     \draw[thick] (K) -- (J);
             \draw[thick] (J) -- (I);

        % Place points on the corners of the square
        \fill ({-0.5*\squareSide},{0.5*\squareSide}) circle (2pt);  % Top left
        \fill ({0.5*\squareSide},{0.5*\squareSide}) circle (2pt);   % Top right
        \fill ({0.5*\squareSide},{-0.5*\squareSide}) circle (2pt);  % Bottom right
        \fill ({-0.5*\squareSide},{-0.5*\squareSide}) circle (2pt); % Bottom left

    \end{scope}
\end{tikzpicture}
\caption{$A,C,E$ are the harmonic points of the equilateral triangle $BDF$.
\\  $H, J, L$ are the harmonic points of the isosceles right triangle $GIK$. }
    \label{fig:harmonic}
\end{figure}
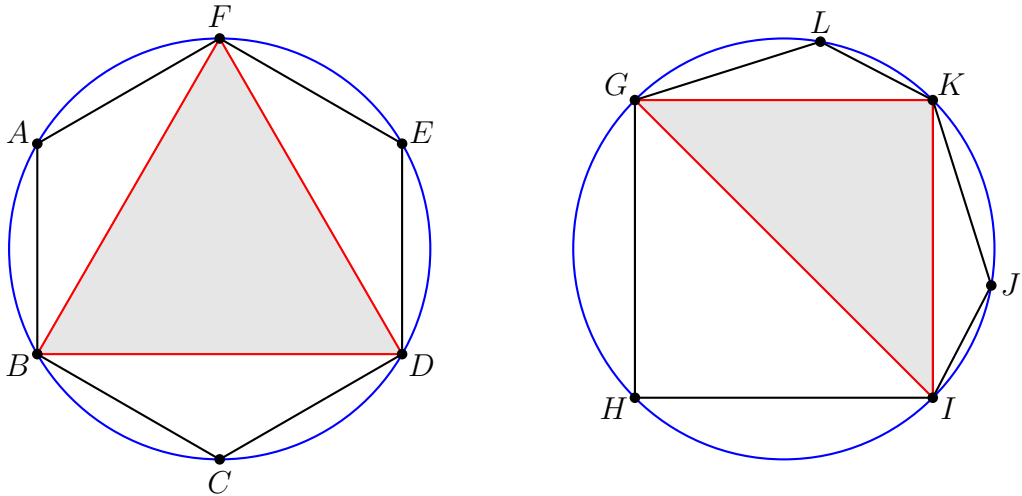

\iffalse
\begin{figure}[h]
    \centering
    \captionsetup{justification=centering}
    \includegraphics[width=\textwidth]{JSpic1.pdf}
    \caption{$A,C,E$ are the harmonic points of the equilateral triangle $BDF$.\\  $H, L, M$ are the harmonic points of the isosceles right triangle $IJG$. }
    \label{fig:harmonic}
\end{figure}
\fi
By applying Theorem~\ref{thm:mainC5}, we prove the following structural result about points sets with many similar copies of a given triangle $T$. Given a triangle $ABC$, call $A, B, C$, its vertices.  Say that a point set {\em contains} a triangle if it contains the three vertices of the triangle.

\begin{theorem}\label{thm:geom}
Let $T$ be a triangle and $S$ be a set of $n > 10^6$ points in the plane such that $S$ contains at least $ m =60\,  n^{11/6}$ triangles similar to $T$. Then there are two triangles $T_1, T_2$ in $S$ (similar to $T$),  and a point $P$ (not necessarily in $S$) such that $P$ is a harmonic point
of both $T_1$ and $T_2$. Moreover, the exponent $11/6=1.8\dot{3}$ cannot be reduced to anything less than $ \approx 1.726$. 

\end{theorem}
%\dm{ can you say something about why harmonic points are particularly nice or useful, maybe it is just worth giving the connection to Katz Tao in detail? otherwise harmonic points seem a bit weird for the nongeometer}

Our proof actually gives  $\Omega(m^6/n^{11})$ pairwise vertex disjoint triangles that all share a common harmonic point. This observation leads to a stronger structural result as the number of triangles $m$ increases. According to the Laczkovich-Ruzsa theorem, for any quadrangle with algebraic cross-ratio, there are point sets with quadratically many copies of quadrangles similar to it. Such sets also have many triangles similar to a triple of the four points of the quadrangle. We show a reverse statement that any set with quadratically many triangles similar to a given triangle $T$ is part of a point set with many quadrangles. 

\begin{theorem}\label{thm:densetriangles}
For every $c>0$ and $\varepsilon>0$, there is a $D>0$ such that the following holds for large enough $n$.
Let $T$ be a triangle and $S$ be a set of $n$ points in the plane such that $S$ contains at least $cn^2$ triangles similar to $T$. Then, there is a quadrangle $Q$ and a set $U$ of at most $Dn$ points such that $U$ contains at least $(c-\varepsilon)n^2$ quadrangles similar to $Q$ and $S\subset U$.
    \end{theorem}

The exponent $11/6$ in Theorem~\ref{thm:geom} also appears in a seemingly unrelated problem investigated by Katz and Tao \cite{KT}. This is no accident, as their question (at least for real or complex numbers) can also be translated to the geometric problem above using the fact that the harmonic point of the triple $a,b,(a+b)/2$ is $(a-b)/2$. We omit the discussion of the arithmetic question of Katz and Tao here but note that improvements in Theorem \ref{thm:geom} would imply improvements in their bound.

\section{Proof of Theorem~\ref{thm:mainC5}}

Given a hypergraph $H$, write $e(H)$ and $d(H)$ for the number of edges and average degree of $H$, respectively.

\begin{proof}  For each vertex $u \in V(H)$,  define the graph $G_u$ as follows: $V(G_u) =V(G)\setminus\{u\}$ and 
$$E(G_u) = \{yz: \exists w,x \hbox{ such that } uwx, xyz \in E(H)\}.$$
Note that the linearity of $H$ implies that the vertices $u,w,x,y,z$ above are all distinct. Another way to define $E(G_u)$ is that it is the set of pairs $yz \in \partial H$ such that there exist distinct edges $e,f \in E(H)$ with $|e \cap f|=1$, $\{y,z\} \subset f\setminus e$,  and  $u \in e\setminus f$. Put differently, $yz \in E(G_u)$ iff there is a linear two-edge path with edges $e,f$ in $H$ starting at $u$ and ending at $yz$ (see Figure 1).

Write $d:=d(H)= 3m/n> 300 \sqrt n$ for the average degree in $H$.
Observe that 
\begin{equation} \label{eqn:Gvlower} \sum_{v \in V(H)} e(G_v) = 4\sum_{x\in V(H)}{d(x) \choose 2} \ge 4n {d \choose 2} \ge 10^6 n^2.
\end{equation}
To see the equality, note that $\sum_{x\in V(H)}{d(x) \choose 2}$ is the number of pairs of edges $e,f$ in $H$ with $|e \cap f|=1$. Writing $e=abx$ and $f=a'b'x$ we see that $ab \in E(G_{a'}) \cap E(G_{b'})$
and $a'b' \in E(G_{a}) \cap E(G_{b})$. Hence the pair $\{e,f\}$ contributes 4 to $\sum_{v \in V(H)} e(G_v)$ and this yields the equality in (\ref{eqn:Gvlower}). The first inequality in (\ref{eqn:Gvlower}) follows from the convexity of binomial coefficients, and the last inequality follows from $d>300\sqrt n$.

Say that a path $wxyz$ in $G_u$ with edges $wx, xy, yz$ is a {\em good path} if there are distinct vertices $a,b,c,a', c'$ in $V(H)$, such that
\begin{equation} \label{eqn:abc}
\{a,b,c,a',c'\} \cap \{u,w,x,y,z\}= \emptyset
\end{equation}
and the following six edges all lie in $E(H)$:
\begin{equation} \label{eqn:edges}
uaa', awx, xyb, ubb', yzc, ucc'.
\end{equation}
We note that if we exclude $ubb'$, the remaining five edges above form a $C_5$. Indeed, this $C_5$ is an expansion of $uaxyc$ (see Figure~\ref{figlower}).

\iffalse
\begin{figure}
\begin{center}
\includegraphics[width=400pt]{pentagon1.pdf}
  \caption{A good path $wxyz$ in $G_u$}\label{figlower}
 \end{center}
\end{figure}
\dm{ nice .. put small dots
at vertices ..}
\fi

\begin{figure}[ht]
\centering
\begin{tikzpicture}[scale=1.5, every node/.style={font=\small}]
  % 1) Define coordinates for each vertex
  % -- Bottom row
  \coordinate (w) at (0.5,0.5);
  \coordinate (x) at (2,0);
  \coordinate (y) at (4,0);
  \coordinate (z) at (5.5,0.5);

  % -- Middle row
  \coordinate (a) at (1,2);
  \coordinate (b) at (3,2);
  \coordinate (c) at (5,2);

  % -- Just above the middle
  \coordinate (a') at (1,3);
  \coordinate (b') at (3.3,3.2);
  \coordinate (c') at (5,3);

  % -- Top vertex
  \coordinate (u)  at (3,4.3);

  % 2) Fill and outline each triangular hyperedge
  % (u,a',a), (u,b,b'), (u,c,c'), (a,w,x), (b,x,y), (c,y,z)
  \fill[blue!20] (u) -- (a') -- (a) -- cycle;
  \draw           (u) -- (a') -- (a) -- cycle;

  \fill[gray!20] (u) -- (b)  -- (b') -- cycle;
  \draw           (u) -- (b)  -- (b') -- cycle;

  \fill[blue!20] (u) -- (c') -- (c)  -- cycle;
  \draw           (u) -- (c') -- (c)  -- cycle;

  \fill[blue!20] (a) -- (w)  -- (x)  -- cycle;
  \draw           (a) -- (w)  -- (x)  -- cycle;

  \fill[blue!20] (b) -- (x)  -- (y)  -- cycle;
  \draw           (b) -- (x)  -- (y)  -- cycle;

  \fill[blue!20] (c) -- (y)  -- (z)  -- cycle;
  \draw           (c) -- (y)  -- (z)  -- cycle;

  % 3) The red pentagon (u,a,x,y,c,u)
  \draw[red, thick] (u) -- (a) -- (x) -- (y) -- (c) -- cycle;

  % 4) Label vertices outside the triangles
  \node[above]        at (u)  {\(u\)};
  \node[left]         at (a') {\(a'\)};
  \node[left]         at (a)  {\(a\)};
  \node[left]   at (b)  {\(b\)};
  \node[right]        at (b') {\(b'\)};
  \node[right]        at (c') {\(c'\)};
  % Adjust c so it is outside (u,c,c') and also not inside (c,y,z)
  % "above left" usually helps keep it off to the left edge.
  \node[left=2pt] at (c)  {\(c\)};
  \node[left]         at (w)  {\(w\)};
  \node[below]        at (x)  {\(x\)};
  \node[below]        at (y)  {\(y\)};
  \node[right]        at (z)  {\(z\)};

% Put points at corners of triangles
   \fill (3, 4.3) circle (2pt);
  \fill (5, 3) circle (2pt);
  \fill (3.3, 3.2) circle (2pt);
  \fill (1, 3) circle (2pt);
    \fill (5, 2) circle (2pt);
      \fill (3, 2) circle (2pt);
        \fill (1, 2) circle (2pt);
          \fill (5.5, 0.5) circle (2pt);
            \fill (4, 0) circle (2pt);
              \fill (2, 0) circle (2pt);
      \fill (0.5, 0.5) circle (2pt);

\end{tikzpicture}
\caption{A good path $wxyz$ in $G_u$ and expansion of $uaxyc$}
\label{figlower}
\end{figure}
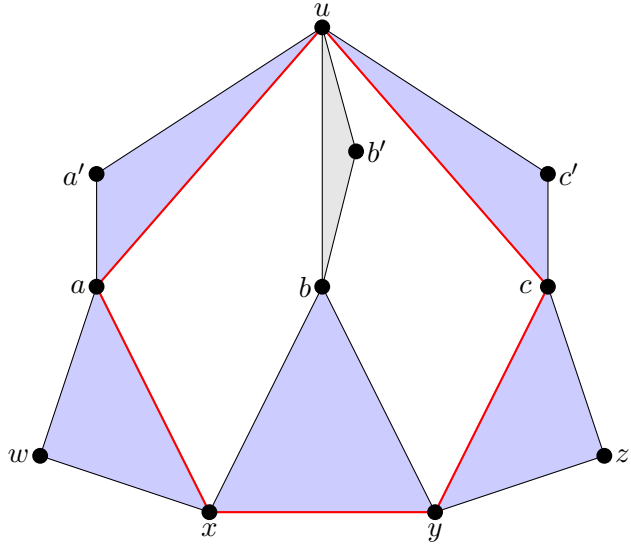

Write $p_v$ for the number of good paths in $G_v$. Since each good path gives rise to a $C_5$ and each $C_5$ is counted at most five times, we conclude that the number of $C_5$s in $H$ is at least 
$\sum_v p_v/5$.
Next, we will obtain a lower bound on $\sum_v e(G_v)$, which will in turn give a lower bound on $\sum_v p_v$.

If $e(G_v)$ is small, then $p_v=0$ is possible, and this is not helpful for us, so we say that $v$ is {\em useful} if $e(G_v) > 1000 n$ and  $v$ is {\em useless} if $e(G_v) \le 1000 n$. Note that (\ref{eqn:Gvlower}) shows
\begin{equation} \label{eqn:useless}
\sum_{v \, useless} e(G_v) \le 1000n^2 < 10^{-3}\sum_{v \in V(H)} e(G_v)\end{equation}
so most of the contribution to $\sum_{v \in V(H)} e(G_v)$ comes from useful $v$ and our plan is to lower bound $\sum_v p_v$ where the sum is over all useful $v$. To this end, let us fix a useful $v$ and consider $G_v$ and $p_v$.  
First, it is necessary to pass to a subgraph of $G_v$ with a large minimum degree, so let $G_v'$ be the subgraph of $G_v$ that remains after iteratively deleting vertices of degree at most 100. As $e(G_v) >1000n$, we have
 \begin{equation} \label{eqn:egv'}e(G_v') \ge e(G_v) - 100n > 0.9 e(G_v).
 \end{equation}
 Call a 3-edge path in $G_v'$ a {\em bad path} if it is not a good path and let $b_v'$ be the number of bad paths in $G_v'$. Our main claim is the following.
 
 {\bf Claim.} $$b_v' \le \sum_{xy \in E(G_v')} 12(d_{G_v'}(x) + 
d_{G_v'}(y)-2).$$ 
{\it Proof of Claim.}
We count bad paths of the form $wxyz$ in $G_v'$ as follows: first, we choose the middle edge $xy$ and then vertices $w$ and $z$ such that $wx$ and $yz$ are both in $E(G_v)$. The definition of $G_v$ gives us (not necessarily distinct) vertices $a,b,c,a', b', c'$  and  (not necessarily distinct)   edges as in (\ref{eqn:edges}) (see Figure 2 for an example where the vertices and edges are distinct).
So we must upper bound the number of $w,z$ such that $a,b,c,a', c'$ are not all distinct or that  (\ref{eqn:abc}) fails. First, we upper bound the number of $\{w,z\}$ such that $b \in \{c,c'\}$. Given any edge $wx \in E(G_v)$, the number of $z$ such that $b \in \{c,c'\}$ is at most two due to linearity of $H$. Indeed, if we have three such distinct vertices, $z, z', z''$ then for two of them, say $z$ and $z'$, vertex $b$ coincides with $c$ or  $b$ coincides with $c'$. If $b$ and $c$ coincide, then the pair $yb=yc$ lies in two distinct edges $ybz$ and $ybz'$, and if $b$ and $c'$ coincide, then the pair $vb=vc'$ lies in two distinct edges; in either case, this contradicts the linearity of $H$.
Hence, the number of such bad paths is at most $2(d_{G_v}(x)-1)$. Arguing similarly for $x$, we deduce that the number of bad paths such that $b \in \{a,c,a',c'\}$ is at most $2(d_{G_v}(x)+d_{G_v}(y)-2)$.

Next we consider bad paths such that $\{a,a'\} \cap \{c,c'\} \ne\emptyset$ and
$b \not\in \{a,c,a',c'\}$. For each choice of $w$, the number of choices of $z$ such that the corresponding vertex $c$  lies in $\{a,a'\}$ is at most two by the linearity of $H$. Hence the number of  bad paths with $c \in \{a,a'\}$ is at most $2(d_{G_v}(x)-1)$. We argue similarly if $c$ is replaced by $c'$ and if the roles of $z$ and $w$ are interchanged. We conclude that  the number of bad paths with $a,c,a',c'$ not all distinct is at most $4(d_{G_v}(x)+d_{G_v}(y)-2)$. As we have assumed $b \not\in \{a,c,a',c'\}$, the number of bad paths with $a,b,c,a',c'$ not all distinct is at most $6(d_{G_v}(x)+d_{G_v}(y)-2)$.

We now consider bad paths containing $xy$ for which $a,b,c,a',c'$ are all distinct that fail (\ref{eqn:abc}). Given a choice of $w$, and hence of distinct $a,a',b$, the number of $z \in \{w,a,a',b\}$ is at most four, since if there are five such distinct $z$, then two of them 
coincide with one of $\{w, a, a', b\}$, which is impossible. Hence  the number of bad paths such that $a,b,c,a',c'$ are all distinct and $z  \in \{w,a,a',b\}$ is at most 
$4(d_{G_v}(x)-1)$.  Similarly, the number of $z$ such that 
$w \in \{c,c'\}$ is at most $2(d_{G_v}(x)-1)$. Reversing the roles of $w$ and $x$ we obtain that the number of bad paths containing $xy$ for which $a,b,c,a',c'$ are all distinct that fail (\ref{eqn:abc}) is at most $6(d_{G_v}(x)+d_{G_v}(y)-2)$. Altogether, the number of bad paths containing $xy$ is at most $12(d_{G_v'}(x) + 
d_{G_v'}(y)-2)$ and the proof of the claim is complete. \qed

Let $s'_v$ be the number of 3-edge paths in $G_v'$. Then
$$s'_v \ge \sum_{xy \in E(G_v')} (d_{G_v'}(x) -2) 
(d_{G_v'}(y)-2)$$ 
as we count paths by picking a neighbor $w$ of $x$ that is not $y$ and then a neighbor $z$ of $y$ that is not $w $ or $x$.
Assume by symmetry that $d_{G_v'}(x)\ge d_{G_v'}(y)$. As the minimum degree in $G_v'$ is at least 100, 
\begin{align*}(d_{G_v'}(x) -2) 
(d_{G_v'}(y)-2) &\ge \frac{d_{G_v'}(x) 
+d_{G_v'}(y)-4}{2}(d_{G_v'}(y)-2) \\
&\ge 49 (d_{G_v'}(x) +d_{G_v'}(y)-4) \\ &> 48 (d_{G_v'}(x) + 
d_{G_v'}(y)-2).
\end{align*}
Consequently, the Claim implies that 
$$s_v'\ge \sum_{xy \in E(G_v')} (d_{G_v'}(x) -2) 
(d_{G_v'}(y)-2)
> \sum_{xy \in E(G_v')}48\,(d_{G_v'}(x) + 
d_{G_v'}(y)-2)
\ge 4b_v'.$$ Write $p'_v$ for the number of good paths in $G_v'$. Then $s_v'=p_v'+b_v'$, so
\begin{equation} \label{eqn:pvsv} p_v \ge p_v' = s_v'-b_v' \ge (0.75) s_v'.\end{equation}
 The number of 3-edge paths in an $n'$ vertex graph with $e'$ edges and average degree $d' = 2e'/n' >100$ is at least  
\begin{equation} \label{eqn:3paths} \frac{(e')^3}{10n'^2}.
\end{equation}
Indeed, to see this, first iteratively delete vertices of degree at most $d'/4$ until no such vertices remain. The remaining graph has at least $e'-n'd'/4 = e'/2$ edges. Now pick an edge and then a neighbour of each of its endpoints to get at least $(e'/2)(d'/4-2)^2>(e')^3/(10n'^2)$ paths.  

Recall from (\ref{eqn:egv'}) that $G_v'$ is a graph with $n' \le n$ vertices and at least $(0.9)e(G_v) \ge 900n$ edges, where the last inequality holds because $v$ is useful.  
Hence by (\ref{eqn:pvsv}) and
(\ref{eqn:3paths}),
$$p_v \ge (0.75) s_v' 
\ge (0.75) \frac{e(G_v')^3}{10n^2} \ge 
(0.75) \frac{(0.9)^3 e(G_v)^3} { 10n^2} > \frac{ e(G_v)^3} { 20n^2}.$$
From (\ref{eqn:useless}), (\ref{eqn:Gvlower}) and $d=3m/n$, we obtain
$$\sum_{v \, usefull} e(G_v) \ge (0.99)\sum_{v \in V(H)} e(G_v)  \ge (0.99) nd^2 > \frac{8m^2}{n}.$$
Hence, by convexity,
$$\sum_{v\in V(H)} p_v \ge \sum_{v \, usefull} p_v \ge
\sum_{v \, usefull} \frac{ e(G_v)^3} { 20n^2}
\ge \frac{1}{20n} \left(\frac{\sum_{v \, usefull}e(G_v)}{n}\right)^3
> \frac{5 m^6}{n^7}.
$$
The number of $C_5$ in $H$ is at least $(1/5)\sum_v p_v$, so the proof is complete.
\end{proof}

\subsection{Constructions of pentagon-free triple systems}\label{C5Free}

As mentioned earlier, the bound $m\gg n^{3/2}$ in Theorem~\ref{thm:mainC5} is sharp. Indeed, Kostochka, the first author and Verstra\"ete  constructed a linear $n$-vertex 3-graph with $\Omega(n^{3/2})$ edges and no $C_5$. We present this construction below as it has not been published before.

{\bf Construction:} Let $T_3(n)$ be the complete 3-partite graph with parts $X, Y, Z$ each of size $n$.
Form the 3-partite linear triple system $H$ of $T_3(n)$ as follows: 
$$V(H) = (X \times Y) \cup (Y \times Z) \cup (X \times Z)$$ 
$$E(H) = \{\{xy, yz, xz\}: (x,y,z) \in X \times Y \times Z\}.$$
Observe that $N:=|V(H)| = 3n^2$ and $|E(H)|=n^3 = (N/3)^{3/2}$. Clearly, $H$ is linear as any two vertices of $H$ that lie in an edge $e$ of $H$ uniquely determine the third vertex of $e$. For example, $xy$ and $yz$ uniquely determine $xz$. 

Next, we prove that $H$ contains no $C_5$. Here, it is convenient to view $E(H)$ as a set of vectors in $\mathbb R^3$ of the form $(x,y,z)$ (rather than sets of the form $\{xy, yz, xz\}$) and use geometric arguments.
Now suppose that there is a $C_5$ with edges $e_1, e_2, \ldots, e_5$ in cyclic order. This means that $e_i \cap e_{i+1}$ are all distinct of size one, and there are no other intersections among the $e_i$s. The list $e_1, \ldots, e_5$ gives rise to a closed walk $W$ of length five in the 3-dimensional grid $\mathbb Z^3$. If some two vertices $v,w$ of $W$ differ in all three coordinates, then the distance between them on $W$ is at least three, so it is impossible to go from $v$ and $w$ and then back in five steps. Hence, we may assume that $W$ is planar and no two consecutive edges of $W$ are in the same axis (as this corresponds to three edges $e_i, e_{i+1}, e_{i+2}$ that all share the same vertex). However, any such closed walk in a planar grid must have an even length.
We conclude that $H$ contains no $C_5$.
\qed

We conjecture below that Theorem~\ref{thm:mainC5} is tight.

\begin{conjecture} \label{conj}
    For $n^{3/2} \ll m \ll n^2$, there exists an $n$-vertex  linear 3-graph in which the number of copies of $C_5$ is $O(m^6/n^7)$.
\end{conjecture}

    As mentioned in the introduction, Proposition~\ref{thm:evidence}
    provides some evidence for Conjecture~\ref{conj}.

   {\bf Proof of Proposition~\ref{thm:evidence}.}
   Recall that we are given a linear $H$ on $n$ vertices and  $m=\Theta(n^{3/2})$ edges, with maximum degree $O(n^{1/2})$, the number of $C_3$ in $H$ is $O(n^{3/2})$, and for  $4\le k\le 5$, the number of $C_k$ in $H$ is $O(n^2)$.  We let $H(t)$ be the 3-graph obtained from $H$ by replacing each vertex of $H$ by a $t$-set of vertices and by replacing each edge of $H$ by a linear 3-partite 3-graph with $t$ vertices in each part and $t^2$ edges. Then $H(t)$ has $N=nt$ vertices and $M = mt^2$ edges. Moreover, it is a short exercise to see that the degree two  vertices of each  $C_5$ in $H(t)$ come from the following three structures in $H$:
   \begin{enumerate}
       \item degree two vertices of $C_3$'s in $H$ together with an additional edge intersecting the $C_3$ 
       \item degree two vertices of $C_k$'s in $H$ for $k=4$ or $k=5$
       \item paths of length at most two in $H$.
   \end{enumerate} The number of $C_5$s that arise from $C_3$ plus edges is by hypothesis $O(n^{3/2}n^{1/2}t^5)= O(n^2t^5)$,  the number of $C_5$s that arise from $C_k$ for $k=4,5$ is  $O(n^2t^5)$ and
   the number of $C_5$s arising from single edges and two edge paths is $O(mt^5) +O(n^2t^5) = O(n^{2}t^5)$. So the total number of $C_5$ in $H(t)$ is at most 
    $O(n^2 t^5) = O(M^6/N^7)$. 
\qed

We point out that our construction $T_3(n)$ has the required properties in the hypothesis of Proposition~\ref{thm:evidence} except that the number of $C_4$ is $\Theta(n^{5/2})$.  Indeed, $T_3(n)$ is linear with no copies of $C_3$ and $C_5$.

\section{Generalization to longer cycles}
In this section, we generalize Theorem~\ref{thm:mainC5} to longer cycles and show a connection to removal lemmas.  The {\em shadow graph} $G$ of a triple system $H$ is defined as follows:
$$V(G)=V(H) \qquad \hbox{  and }  \qquad E(G)= \partial H = \{yz: \exists x \hbox{ with } xyz \in E(H)\}.$$
A linear path is a 3-graph obtained from  a linear cycle by deleting exactly one edge. Given a linear path $P$ in a 3-graph, say that a vertex is an endpoint of $P$ if it lies in the first or last edge of $P$ and it has degree one on $P$. In the theorem below, we use asymptotic notation and assume, wherever needed, that $n$ is sufficiently large. In particular, $a \gg b$ means that $a>Cb$ for some sufficiently large constant $C>0$.
\iffalse
However, our result only applies to the larger class of Berge cycles. A {\em Berge} cycle of length $k$, written $BC_{k}$ is a 3-graph with distinct vertices $v_1, \ldots, v_{k}$ and distinct edges $e_1, \ldots, e_k$, where $e_i \supset \{v_i, v_{i+1}\}$ and indices are taken modulo $k$. For our application, it suffices to consider a  relaxation of 
$BC_k$ defined as follows: A {\em shadow Berge cycle} of length $k$, denoted $BC'_k$, is a berge k-cycle where the edges $e_1, \ldots, e_k$ need no be distinct.
\fi

\begin{theorem} \label{thm:ck}
    Let $k \ge 2$ be an integer and let $H$ be an $n$-vertex linear triple system with $m \gg n^{2-1/3k}$ edges. Then the shadow graph $G$ of $H$ contains at least $m^{3k}/n^{4k-1}$ copies of $C_{2k+1}$.
\end{theorem}
\begin{proof}
For each vertex $u \in V(H)$,  define the multigraph $G_u$ as follows. Let $V(G_u) =V(H)\setminus\{u\}$. Next, let $P$ be a $k$-edge linear path $e_1, \ldots, e_k$ in $H$ with endpoint $u \in e_1$ and $e_k=xyz$ where $y,z$ are endpoints of $P$. Then the edge $e_u(P)=yz$ is an edge of $G_u$. We emphasize that $G_u$ is a multigraph; indeed, the pair $yz$ can arise many times in $E(G_u)$ due to many paths $P$ from $u$, and we distinguish the edges comprising the pair depending on the path (see Figure 3).

\begin{figure}[ht]
\centering
\begin{tikzpicture}[scale=1.7, every node/.style={font=\small}]
  % 1) Define coordinates for each vertex
  % -- Bottom triangle
  \coordinate (y) at (-1,-0.4);
  \coordinate (z) at (1,-0.4);
  \coordinate (x) at (0,0.75);

  % -- Left path
   \coordinate (a) at (-1.25,1.25);
  \coordinate (b) at (-1,2);
  \coordinate (c) at (-1.75,2.5);
  \coordinate (d) at (-1,3);
 \coordinate (e) at (-1.25,3.75);

  % -- Right path
   \coordinate (a') at (1.25,1.25);
  \coordinate (b') at (1,2);
  \coordinate (c') at (1.75,2.5);
  \coordinate (d') at (1,3);
 \coordinate (e') at (1.25,3.75);

  % -- Top vertex
  \coordinate (u)  at (0,4.25);

  % 2) Fill and outline each triangular hyperedge
  % (u,a',a), (u,b,b'), (u,c,c'), (a,w,x), (b,x,y), (c,y,z)
  \fill[gray!20] (x) -- (y) -- (z) -- cycle;
  \draw           (x) -- (y)  -- (z) -- cycle;

  \fill[red!20] (a) -- (b)  -- (x) -- cycle;
  \draw           (a) -- (b)  -- (x) -- cycle;

  \fill[red!20] (b) -- (c) -- (d)  -- cycle;
  \draw           (b) -- (c) -- (d)  -- cycle;

  \fill[red!20] (d) -- (e)  -- (u)  -- cycle;
  \draw           (d) -- (e)  -- (u)  -- cycle;

  \fill[blue!20] (a') -- (b')  -- (x) -- cycle;
  \draw           (a') -- (b')  -- (x) -- cycle;

  \fill[blue!20] (b') -- (c') -- (d')  -- cycle;
  \draw           (b') -- (c') -- (d')  -- cycle;

  \fill[blue!20] (d') -- (e')  -- (u)  -- cycle;
  \draw           (d') -- (e')  -- (u)  -- cycle;

%draw arcs on the bottom
\draw[thick] [red] (y) arc[start angle=125, end angle=55, radius=1.8cm];
\draw[thick] [blue] (y) arc[start angle=-125, end angle=-55, radius=1.8cm];

  % 4) Label vertices outside the triangles
  \node[above=3pt]        at (u)  {\(u\)};
  
  \node[above=3pt]        at (x)  {\(x\)};
  \node[below left]        at (y)  {\(y\)};
  \node[below right]        at (z)  {\(z\)};

  % Label the paths
 \node[left=10pt]        at (c)  {\textcolor{red}{\Large $P$}};
 \node[right=10pt]        at (c')  {\textcolor{blue}{\Large $P'$}};

% Label the edges
\node[]        at (-0.85, 3.6)  {$e_1$};
\node[]        at (-1.25, 2.5)  {$e_2$};
\node[]        at (-0.85, 1.4)  {$e_3$};

\node[]        at (0.85, 3.6)  {$e_1'$};
\node[]        at (1.25, 2.5)  {$e_2'$};
\node[]        at (0.85, 1.4)  {$e_3'$};
\node[]        at (0, 0.19)  {$e_4=e_4'$};

 % Label some vertices 
  \node[above=3pt]        at (x)  {\(x\)};
%  \node[below]        at (y)  {\(y\)};
%  \node[below]        at (z)  {\(z\)};

%say path length
 \node[]        at (0,2.5)  {{\Large $k=4$}};

 %\label edges of Gu
\node[]        at (0, -0.26)  {\textcolor{red}{$e_u(P)$}};
\node[]        at (0, -1)  {\textcolor{blue}{$e_u(P')$}};

% Put points at corners of triangles
   \fill (x) circle (2pt);
 \fill (y) circle (2pt);
  \fill (z) circle (2pt);
   \fill (a) circle (2pt);
    \fill (b) circle (2pt);
     \fill (c) circle (2pt);
      \fill (d) circle (2pt);
          \fill (e) circle (2pt);
       \fill (a') circle (2pt);
        \fill (b') circle (2pt);
         \fill (c') circle (2pt);
          \fill (d') circle (2pt);
           \fill (e') circle (2pt);
            \fill (u) circle (2pt);

\end{tikzpicture}
\caption{The multigraph $G_u$}
\label{figkpath}
\end{figure}
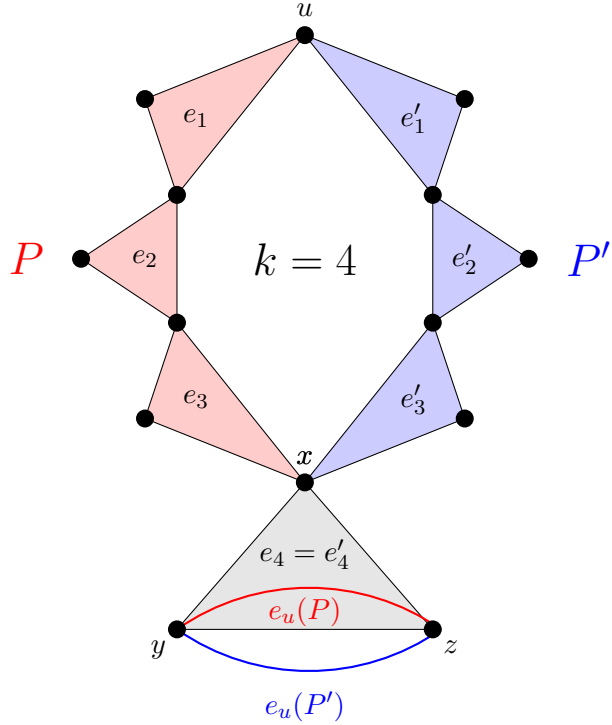

\iffalse
\begin{figure}
\begin{center}
\includegraphics[width=350pt]{pentagon2.pdf}
  \caption{The multigraph $G_u$}\label{figkpath}
 \end{center}
\end{figure}

\fi
Write $e_u=|E(G_u)|$ and $d$ for the average degree of $H$, so that $d \gg n^{1-1/3k}$. The number of  $k$-edge linear paths in $H$ is at least $\Omega(nd^k)$; to see this, let $H'\subset H$ be the 3-graph that remains after iteratively removing vertices of degree at most $d/2$ from $H$, and then build paths by starting with any edge $e_1$ of $H'$ and then greedily choosing edges $e_2, \ldots, e_k$, where we have at least $d/2-2i>d/4$ choices for each $e_i$.

The quantity $\sum_u e_u$ is the number of pairs $(u,e_u(P))$ where $u$ is a vertex and $P$ is a $k$-edge path with endpoint $u$. The number of $k$-edge paths $P$ is $\Omega(nd^k)$ and
each such $P$  gives rise to four pairs $(u,e_u(P))$. A given pair $(u,e_u(P))$ cannot arise from a $k$-edge path different from $P$.  As $d \gg n^{1-1/3k}$,
\begin{equation} \label{eqn:ndk}
\sum_{u \in V(H)} e_u = \Omega(nd^k) \gg n^2.
\end{equation} As before, say that $u$ is {\em useless} if $e_u< 100 n$ and useful otherwise. Then
$\sum_{u\, useless} e_u < 100n^2 \ll  \sum_{u\in V(H)} e_u$, so $\sum_{u\, useful} e_u= 
\Omega(nd^k)$. 

A 3-edge path in the multigraph $G_v$ is a set of four (not necessarily distinct) vertices $v_1, v_2, v_3, v_4$ and three distinct edges $e_1, e_2, e_3$  such that $e_i =v_iv_{i+1}$ for $i=1,2,3$ (the pair might appear with multiplicity greater than one, but the edges are distinct). 
For every useful $v$ the number of 3-edge paths $p_v$ in $G_v$ is at least 
$\Omega (e_v^3/n^2)$. Indeed, since $v$ is useful, $e_v\ge 100n$ and $G_v$ has average degree $d_v \ge 300$, so we restrict to a subgraph of minimum degree at least $d_v/4$ and then build a 3-edge path greedily. There are at least $e_v/4$ choices for the middle edge and at least $d_v/4-2>d_v/5>(3e_v/5n)$ choices for each of the other two edges. 
Consequently, by (\ref{eqn:ndk}), 
\begin{equation} \label{eqn:pvk}\sum_{v \in V(H)} p_v = \Omega\left(\frac{\sum_{v} e_v^3}{n^2}\right) =
\Omega\left(\frac{(\sum_{v} e_v)^3}{n^4}\right)
=\Omega\left(\frac{(nd^k)^3}{n^4}\right)
=\Omega\left(\frac{d^{3k}}{n}\right)
=\Omega\left(\frac{m^{3k}}{n^{3k+1}}\right).
\end{equation}
An $\ell$-{\em pseudocycle} is a homomorphic image of an $\ell$-cycle in $G$. Suppose that 
$wxyz$ is a 3-edge path in $G_v$ with  edges $e_v(P^1)=wx, e_v(P^2)=xy, e_v(P^3)=yz$. Let $P^i=e^i_1, \ldots, e^i_k$ denote the $k$-path in $H$ from $v$ to the edge $e_v(P^i)\in G_v$ for $i=1,2,3$. For $1\le j \le k-1$, let $v^i_j = e^i_j \cap e^i_{j+1}$. 
Each 3-edge path $P$ in $G_v$ with vertices $w,x,y,z$ as above gives rise to the following $(2k+1)$-pseudocycle $C_v$ in $G$ (see Figure~\ref{fig7cycle}). 

\begin{figure}[ht]
\centering
\begin{tikzpicture}[scale=1.7, every node/.style={font=\small}]
  % 1) Define coordinates for each vertex
  % -- Bottom vertices
  \coordinate (w) at (-2,-0.2);
  \coordinate (x) at (-0.66,-0.4);
  \coordinate (y) at (0.66,-0.4);
  \coordinate (z) at (2,-0.2);

  % -- Shifted Left Path (Red)
  \coordinate (a) at (-2.25,1.25);
  \coordinate (b) at (-2,2);
  \coordinate (c) at (-2.75,2.5);
  \coordinate (d) at (-2,3);
  \coordinate (e) at (-2.25,3.75);

  % -- Shifted Right Path (Blue)
  \coordinate (a') at (2.25,1.25);
  \coordinate (b') at (2,2);
  \coordinate (c') at (2.75,2.5);
  \coordinate (d') at (2,3);
  \coordinate (e') at (2.25,3.75);

  % -- Top vertex
  \coordinate (v)  at (0,4.25);

  % -- Middle points for chain triangles
  \coordinate (p1) at (0,2.5); % bottom Vertex of first triangle below u
  \coordinate (q1) at (0.66,3.125); % right Vertex of first triangle below u
  \coordinate (q2) at (-0.66,1.75); % leftVertex of middle green triangle below u
  \coordinate (p2) at (0,1); % Vertex of bottom green triangle connecting to x and y

  % 2) Fill and outline each triangular hyperedge
  % Bottom triangles
  \fill[red!20] (w) -- (x) -- (a) -- cycle;
  \draw           (w) -- (x) -- (a) -- cycle;

  \fill[blue!20] (y) -- (z) -- (a') -- cycle;
  \draw           (y) -- (z) -- (a') -- cycle;

  % Chain of triangles
  \fill[gray!20] (v) -- (p1) -- (q1) -- cycle; % Top triangle
  \draw           (v) -- (p1) -- (q1) -- cycle;

  \fill[gray!20] (p1) -- (p2) -- (q2) -- cycle; % Middle triangle
  \draw           (p1) -- (p2) -- (q2) -- cycle;

  \fill[gray!20] (p2) -- (x) -- (y) -- cycle; % Bottom triangle
  \draw           (p2) -- (x) -- (y) -- cycle;

  % Left triangles
 
  \fill[red!20] (a) -- (c) -- (d) -- cycle;
  \draw           (a) -- (c) -- (d) -- cycle;

  \fill[red!20] (d) -- (e) -- (v) -- cycle;
  \draw           (d) -- (e) -- (v) -- cycle;

  % Right triangles

  \fill[blue!20] (a') -- (c') -- (d') -- cycle;
  \draw           (a') -- (c') -- (d') -- cycle;

  \fill[blue!20] (d') -- (e') -- (v) -- cycle;
  \draw           (d') -- (e') -- (v) -- cycle;

  % 4) Label vertices outside the triangles
  \node[above=3pt]        at (v)  {\(v\)};
  
  \node[below=3pt]        at (w)  {\(w\)};
  \node[below=3pt]        at (x)  {\(x\)};
  \node[below=3pt]        at (y)  {\(y\)};
  \node[below=3pt]        at (z)  {\(z\)};

 \node[right=3pt]        at (d)  {$v_1^1$};
\node[right=3pt]        at (a)  {$v_2^1$};

 \node[left=3pt]        at (a')  {$v_2^3$};
 \node[below left]        at (d')  {$v_1^3$};

 \node[right=2pt]        at (p1)  {$v_2^2$};
 \node[right=2pt]        at (p2)  {$v_1^2$};

  % Label the paths
  \node[left=10pt]        at (c)  {\textcolor{red}{\Large $P^1$}};
  \node[left=10pt]       at (p1) {{\Large $P^2$}};
  \node[right=10pt]       at (c') {\textcolor{blue}{\Large $P^3$}};

%draw arcs on the bottom
\draw[thick] [red] (w) arc[start angle=-125, end angle=-70, radius=1.5cm];
\draw[thick] [blue] (y) arc[start angle=-110, end angle=-55, radius=1.5cm];
\draw[thick] [black] (x) arc[start angle=-110, end angle=-70, radius=2cm];

%draw the 7cycle in the shadow

  \draw[very thick] [black]   (v) -- (d) -- (a) -- (x) -- (y) -- (a') -- (d') -- (v) -- cycle;     
  
  % Label the edges
  \node[]        at (-1.75, 3.6)  {$e_1^1$};
  \node[]        at (-2.35, 2.4)  {$e_2^1$};
  \node[]        at (-1.7, 0.3)  {$e_3^1$};

  \node[]        at (1.75, 3.6)  {$e_1^3$};
  \node[]        at (2.35, 2.4)  {$e_2^3$};
  \node[]        at (1.7, 0.3)  {$e_3^3$};

 \node[]        at (0.3, 3.2)  {$e_1^2$};
  \node[]        at (-0.3, 1.8)  {$e_2^2$};
  \node[]        at (0, 0.3)  {$e_3^2$};
  
%Label the 3 path
 \node[]        at (-1.35, -0.7)  {\textcolor{red}{$e_v(P^1)$}};
  \node[]        at (0, -0.8)  {$e_v(P^2)$};
  \node[]        at (1.35, -0.7)  {\textcolor{blue}{$e_v(P^3)$}};

  % Indicate path length
  \node[]        at (1.25,2)  {{\Large $k=3$}};

%label shadowcycle
\node[]        at (1.2,0.7)  {{\Large $C_v$}};

  % Put points at corners of triangles
  \fill (w) circle (2pt);
  \fill (x) circle (2pt);
  \fill (y) circle (2pt);
  \fill (z) circle (2pt);
  \fill (p1) circle (2pt);
  \fill (p2) circle (2pt);
  \fill (a) circle (2pt);
%  \fill (b) circle (2pt);
  \fill (c) circle (2pt);
  \fill (d) circle (2pt);
  \fill (e) circle (2pt);
  \fill (a') circle (2pt);
%  \fill (b') circle (2pt);
  \fill (c') circle (2pt);
  \fill (d') circle (2pt);
  \fill (e') circle (2pt);
  \fill (v) circle (2pt);
 \fill (q1) circle (2pt);
  \fill (q2) circle (2pt);

\end{tikzpicture}
\caption{A 7-pseudocycle $C_v$}\label{fig9cycle}
\label{fig7cycle}
\end{figure}
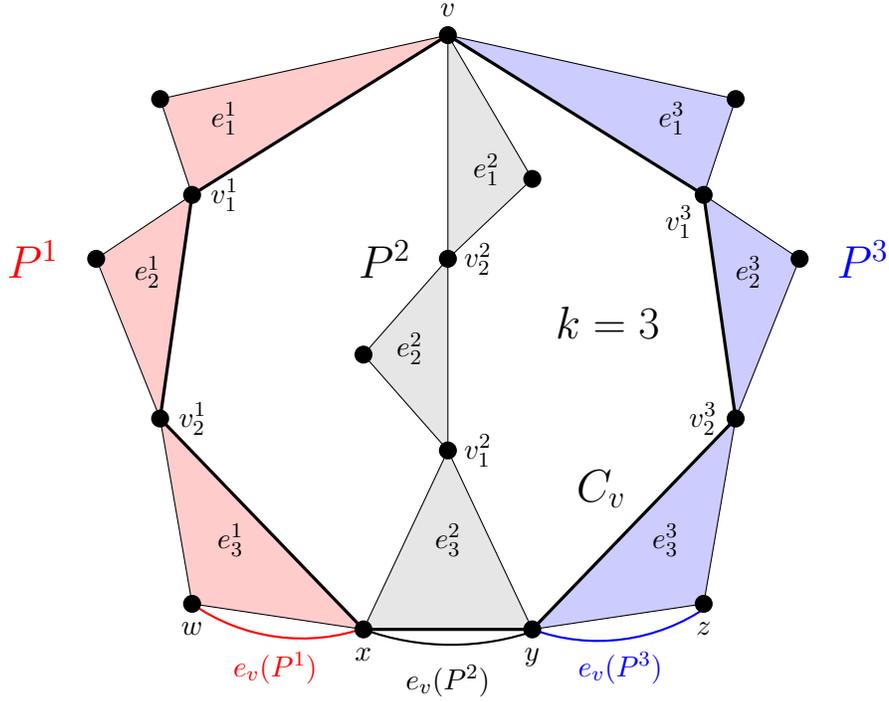

\iffalse
\begin{figure}
\begin{center}
\includegraphics[width=350pt]{pentagon3.pdf}
  \caption{A 7-pseudocycle $C_v$}\label{fig9cycle}
 \end{center}
\end{figure}
\fi
The vertices of $C_v$, in cyclic order, are 
$$v, v^1_1, v^1_2, \ldots, v^1_{k-1}, x, y, v^3_{k-1}, v^3_{k-2}, \ldots, v^3_1.$$
We emphasize that $C_v$ is a $(2k+1)$-pseudocycle as vertices can be repeated.

Given $v$ and $e=xy$, the number of $k$-edge paths $P^2$ in $H$  starting at $v$ and ending at $e=e_v(P^2)$ is at most $n^{k-2}$. This is because there is at most one choice for the vertex $v^2_{k-1}$ as $H$ is linear,  there are at most $k-2$ other vertices of degree two on $P^2$, and once these are chosen, the path $P^2$ is determined again due to linearity of $H$. Hence each $(2k+1)$-pseudocycle $C_v$ obtained from $wxyz$ is counted at most $n^{k-2}$ times. 
Consequently, by (\ref{eqn:pvk}), the number of  $(2k+1)$-pseudocycles in $G$ is at least 
$$\frac{\sum_{v \in V(H)} p_v}{n^{k-2}} = \Omega\left(\frac{m^{3k}}{n^{4k-1}}\right).$$
The number of these $(2k+1)$-pseudocycles with fewer than $2k+1$ vertices is at most $n^{2k} \ll   m^{3k}/n^{4k-1}$ as $m \gg n^{2-1/3k}$. Hence the number of copies of $C_{2k+1}$ in $G$ is at least $m^{3k}/n^{4k-1}$ by adjusting the constant in the hypothesis  $m\gg n^{2-1/3k}$.
\end{proof}
\bigskip
We remark that with more care, Theorem~\ref{thm:ck} can be extended to find Berge cycles in $H$ instead of just cycles in the shadow graph $G$ (the additional requirement is that the edges of the cycle are distinct). We wrote the technically simpler argument that finds only cycles in the shadow graph as it suffices for the application below, which restates Theorem~\ref{thm:epsfar}.

\begin{corollary}
    Fix $k \ge 2$. There is a constant $c$ such that if an $n$-vertex graph $G$ is $\varepsilon$-far from being triangle-free, with $\varepsilon \gg n^{-1/3k}$, then $G$ has at least $c \, \varepsilon^{3k} n^{2k+1}$ copies of $C_{2k+1}$. 
\end{corollary}
\begin{proof}
    Let $H$ be a maximal collection of edge-disjoint triangles in $G$. View $H$ as a 3-graph whose edges are the triangles. Because the triangles in $G$ are edge-disjoint,  $H$ is linear. Moreover, if $H$ has $m$ (hyper)edges, then by maximality, we can delete $3m$ edges in $G$ so that the resulting graph is triangle-free. As $G$ is $\varepsilon$-far from being triangle-free,  $m \ge \varepsilon n^2/3 \gg n^{2-1/3k}$. Since $G$ contains the shadow graph of $H$, by Theorem~\ref{thm:ck}, the number of $C_{2k+1}$ in $G$ is  at least $\Omega(m^{3k}/n^{4k-1}) = \Omega(\varepsilon^{3k} n^{2k+1})$.
\end{proof}

\section{Proof of Theorem~\ref{thm:geom}}
In this section, we use Theorem~\ref{thm:mainC5} to prove Theorem~\ref{thm:geom}. Say that a triangle lies in a set if its three vertices are in the set. Suppose $n>10^6$ and $S$ is a set of $n$ points and there are $m>60 n^{11/6}$ triangles in $S$ similar to a given triangle $T=(A,B,C)$. Partition $S$ randomly into three sets, $V_A,V_B,V_C$, where we place each point of $S$ into one of the sets with equal probability 1/3. The expected number of  triangles $A'B'C'$ in $S$ similar to $T$ with $A'\in V_A,B'\in V_B,C'\in V_C$ such that there is a similarity transformation $A'B'C'\to ABC$ with $A \to A', B \to B', C \to C'$ is $m/27$. Therefore, there is a particular choice of  $V_A,V_B,V_C$ such that the number of triangles $A'B'C'$  as above is at least $m/27$. We will also need the family of similar triangles to have the same orientation. There are at least $$m'\ge m/54> n^{11/6}> 100 n^{3/2}$$ such triangles.

Let $H$ be the 3-partite 3-graph where  $V(H)=S$ and $E(H)$ is the set of triangles in $S$ similar to $T$. Then $H$ is linear with $n$ vertices and $m'>100n^{3/2}$ edges,  so by Theorem~\ref{thm:mainC5}, the number of linear $C_5$'s (henceforth pentagons) in $H$ is at least
\begin{equation}\label{Quant}
\frac{m'^6}{n^7} > \frac{n^{11}}{n^7} =  n^4.
\end{equation}
The cycle of a pentagon is the (unique) graph cycle in the shadow graph of the pentagon. Every pentagon $P$ in $H$ has one degree-two vertex of its cycle $C$ in one of the three vertex classes and two degree-two vertices in each of the remaining two vertex classes.  For a given pentagon $P$, suppose that $V_A$ and $V_B$ have two degree-two vertices and $V_C$ has one degree-two vertex (See Figure \ref{fig:regular_pentagon}).

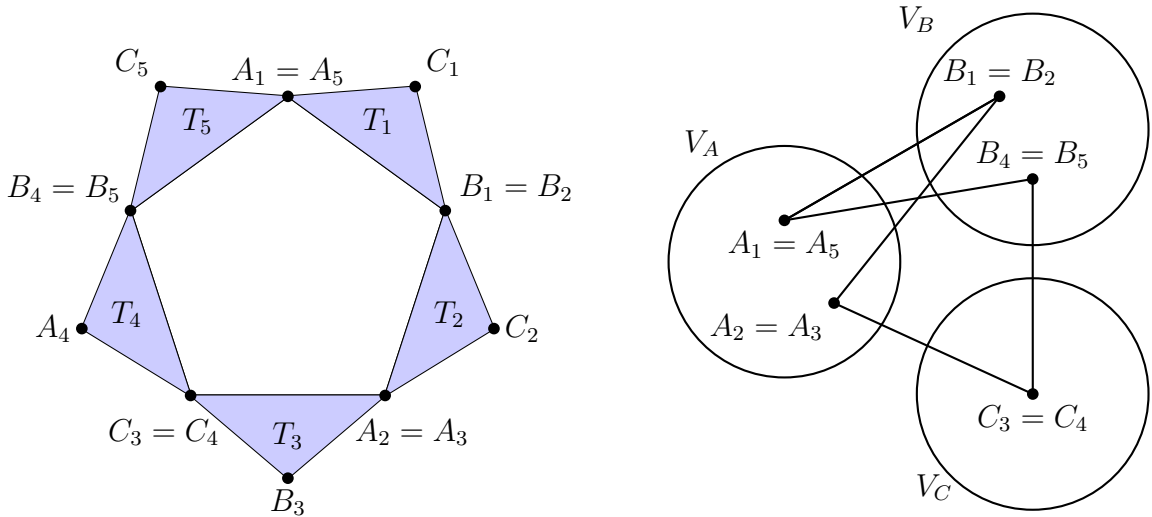
\begin{figure}[ht]
\centering
\begin{tikzpicture}[scale=1.1]
  % Draw a regular pentagon with the base horizontal and radius 2cm
  \draw[thick] (18:2cm) 
    \foreach \x in {90, 162, 234, 306} 
    { -- (\x:2cm) } 
    -- cycle;

 % Define the vertices of the pentagon
  \coordinate (X1) at (18:2.cm);     % Top
  \coordinate (X2) at (90:2cm);    % Top left
  \coordinate (X3) at (162:2cm);    % Bottom left
  \coordinate (X4) at (234:2cm);    % Bottom right
  \coordinate (X5) at (306:2cm);     % Top right
    
  % Label the vertices in clockwise order using math mode, starting from the top
  \node at (18:2.9cm) {$B_1=B_2$}; % Label for the first vertex (left)
  \node at (90:2.3cm) {$A_1=A_5$}; % Label for the second vertex (top right)
  \node at (162:2.85cm) {$B_4=B_5$}; % Label for the third vertex (bottom right)
  \node at (234:2.55cm) {$C_3=C_4$}; % Label for the fourth vertex (bottom left)
  \node at (306:2.55cm) {$A_2=A_3$}; % Label for the fifth vertex (left)

  % Find C_1, which should be on the perpendicular bisector of A_1A_2
  \coordinate (C1) at (54:2.618cm); % Placing C1 such that A_1A_2C_1 is equilateral
  \node at (C1) [above right] {$C_1$}; % Label for C_1

  % Draw the equilateral triangle A_1A_2C_1
  \draw[thick] (18:2cm) -- (90:2cm) -- (C1) -- cycle;

  % Find C_2, which should be on the perpendicular bisector of A_2A_3
  \coordinate (C2) at (126:2.618cm); % Placing C2 such that A_2A_3C_2 is equilateral
  \node at (C2) [above left] {$C_5$}; % Label for C_2

  % Draw the equilateral triangle A_2A_3C_2
  \draw[thick] (90:2cm) -- (162:2cm) -- (C2) -- cycle;

  % Find C_3, which should be on the perpendicular bisector of A_3A_4
  \coordinate (C3) at (198:2.618cm); % Placing C3 such that A_3A_4C_3 is equilateral
  \node at (C3) [left] {$A_4$}; % Label for C_3

  % Draw the equilateral triangle A_3A_4C_3
  \draw[thick] (162:2cm) -- (234:2cm) -- (C3) -- cycle;
  
  % Find C_4, which should be on the perpendicular bisector of A_4A_5
  \coordinate (C4) at (270:2.618cm); % Placing C4 such that A_4A_5C_4 is equilateral
  \node at (C4) [below] {$B_3$}; % Label for C_4

  % Draw the equilateral triangle A_4A_5C_4
  \draw[thick] (234:2cm) -- (306:2cm) -- (C4) -- cycle;
  
  % Find C_5, which should be on the perpendicular bisector of A_5A_1
  \coordinate (C5) at (342:2.618cm); % Placing C5 such that A_5A_1C_5 is equilateral
  \node at (C5) [right] {$C_2$}; % Label for C_5

  % Draw the equilateral triangle A_5A_1C_5
  \draw[thick] (306:2cm) -- (18:2cm) -- (C5) -- cycle;

  % Shade the triangles
  \fill[blue!20] (X1) -- (X2) -- (C1) -- cycle; % Triangle A1A2C1
  \fill[blue!20] (X2) -- (X3) -- (C2) -- cycle; % Triangle A2A3C2
  \fill[blue!20] (X3) -- (X4) -- (C3) -- cycle; % Triangle A3A4C3
  \fill[blue!20] (X4) -- (X5) -- (C4) -- cycle; % Triangle A4A5C4
  \fill[blue!20] (X5) -- (X1) -- (C5) -- cycle; % Triangle A5A1C5

%label triangles
 \node at (C1) [below left=0.2cm] {$T_1$}; % Label for T1
 \node at (342:2.05cm) {$T_2$};
% \node at (C5) [above left=0.2cm] {$T_2$}; % Label for T2
  \node at (C4) [above=0.25cm] {$T_3$}; % Label for T3
   \node at (198:2.05cm) {$T_4$};
%  \node at (C3) [above right=0.2cm] {$T_4$}; % Label for T4
  \node at (C2) [below right=0.2cm] {$T_5$}; % Label for T5
  
  % Place small circles at the vertices \node at (C1) [below left=0.2cm] {$T_1$}; % Label for T1

  \fill (18:2cm) circle (2pt); % Circle at A1
  \fill (90:2cm) circle (2pt); % Circle at A2
  \fill (162:2cm) circle (2pt); % Circle at A3
  \fill (234:2cm) circle (2pt); % Circle at A4
  \fill (306:2cm) circle (2pt); % Circle at A5

\fill (C1) circle (2pt); % Circle at C_1
 \fill (C2) circle (2pt); % Circle at C_2
 \fill (C3) circle (2pt); % Circle at C_3
  \fill (C4) circle (2pt); % Circle at C_4
 \fill (C5) circle (2pt); % Circle at C_5

  % Draw three large circles arranged in a triangular fashion to the right
  \coordinate (P1) at (6, 0);    % First circle (bottom left)
  \coordinate (P2) at (9, 1.6);  % Second circle (top)
  \coordinate (P3) at (9, -1.6);  % Third circle (bottom right, moved closer to the first)
  
  % Draw much larger circles at P1, P2, P3 with radius 1.4cm
  \draw[thick] (P1) circle(1.4cm);  % First circle
  \draw[thick] (P2) circle(1.4cm);  % Second circle
  \draw[thick] (P3) circle(1.4cm);  % Third circle
  
  % coordinatize some  points in the circles
   \coordinate (A15) at (6,0.5);     % Top
    \fill (A15) circle (2pt); % Circle at A1
  \node at (A15) [below] {$A_1=A_5$}; % Label for A_1=A5

 \coordinate (A23) at (6.6,-0.5);     % Top
    \fill (A23) circle (2pt); % Circle at A1
  \node at (A23) [below left] {$A_2=A_3$}; % Label for A_1=A5

 \coordinate (C34) at (P3);     % Top
    \fill (C34) circle (2pt); % Circle at A1
  \node at (C34) [below] {$C_3=C_4$}; % Label for A_1=A5

 \coordinate (B45) at (9, 1);     % Top
    \fill (B45) circle (2pt); % Circle at A1
  \node at (B45) [above] {$B_4=B_5$}; % Label for A_1=A5

 \coordinate (B12) at (8.6, 2);     % Top
    \fill (B12) circle (2pt); % Circle at A1
  \node at (B12) [above ] {$B_1=B_2$}; % Label for A_1=A5
    
   \draw[thick] (A15) -- (B12);
    \draw[thick] (A23) -- (B12);
      \draw[thick] (A23) -- (C34);
        \draw[thick] (C34) -- (B45);
          \draw[thick] (A15) -- (B45);
            \draw[thick] (A15) -- (B12);

  \node at (A15) [above left=1cm] {$V_A$}; % Label for A_1=A5
    \node at (B12) [above left=1cm] {$V_B$}; % Label for A_1=A5
      \node at (C34) [below left=1.3cm] {$V_C$}; % Label for A_1=A5

\end{tikzpicture}
\caption{Triangles forming a pentagon}
\label{fig:regular_pentagon}
\end{figure}

\iffalse
\begin{figure}
\hspace{-50pt}\includegraphics[width=1.2\textwidth]{JSpic2.pdf}
  \caption{Triangles forming a pentagon}\label{JCpic2}
 \end{figure}
\fi
 Denote the five triangles of the pentagon by  $T_1,\ldots, T_5$, in cyclic order,  and the vertices of $T_j$ by $A_j,B_j,C_j$. Then, the vertices of the cycle $C$ of the pentagon $P$ in cyclic order are

\begin{equation}\label{penta_vertices}
 A_1 (=A_5), B_1 (=B_2), A_2 (=A_3), C_3 (=C_4), B_4 (=B_5).   
\end{equation}

Note that the five degree-one points of $P$, in cyclic order, are 
$$C_1, C_2, B_3, A_4, C_5.$$ The first four of theses, $C_1, C_2, B_3, A_4$, are vertices of $T_1, \ldots T_4$, respectively, and $C_5$ is a vertex of $T_5$.
To prove the theorem, we show the following lemma.

\begin{lemma}\label{four_fix}
The four points  $C_1, C_2, B_3, A_4$ in $P$ determine a harmonic point of the fifth triangle $T_5$. 
\end{lemma}

This will complete our proof of Theorem \ref{thm:geom} since we may associate to each pentagon $P$ its four points as in the claim. By pigeonhole, using (\ref{Quant}), there are two pentagons $P, P'$ that are associated to the same four points $C_1, C_2, B_3, A_4$. The fifth triangles $T_5$ of $P$ and $T_5'$ of $P'$ then have the same harmonic points. Moreover, $T_5$ and $T_5'$ have distinct points in $V_A$ and in $V_B$, as any one of these points determines the pentagon if we are also given $C_1, C_2, B_3, A_4$ as the degree one points. Further, the two quadrangles given by the two triangles $T_5$ and $T_5'$ and their common harmonic point are similar, so if they share two vertices (with the same labels), they are the same. Therefore, $T_5$ and $T_5'$ are, in fact, vertex disjoint.

{\bf Proof of Lemma~\ref{four_fix}.}
For the sake of simplicity, the complex number $z_{P}$ is denoted by the point $P$ in the following calculations.  The triangles $T_1, \ldots, T_5$ are similar, so the vertex $C_j$ can be expressed as the following linear combination of $A_j, B_j$ where $z=z(T)$  depends only on $T$: 
\begin{equation}\label{complC}
C_j=\frac{A_j+B_j}{2}+\frac{z(A_j-B_j)}{2}, \quad \text{where} \quad z=re^{i\theta} \end{equation} 
To see that (\ref{complC}) holds, note that 
$$z=\frac{2C_j - (A_j+B_j)}{A_j-B_j}.$$
Multiplying each of $A_j, B_j, C_j$ by $re^{i\alpha}$ clearly leaves $z$ unchanged, which means that dilating and rotating
$A_jB_jC_j$ by a factor $r$ and an angle $\alpha$ preserves $z$. Adding $w=s+ti$ to each of $A_j, B_j, C_j$ also leaves $z$ unchanged. Since every triangle similar to $A_jB_jC_j$ with the same orientation is obtained by dilating, rotating and shifting, $z$ indeed depends only on $T$. 

Using parameter $z$, we can express a harmonic point of $A_5B_5C_5$ as

\[
D_5=\frac{A_5+B_5}{2}+\frac{A_5-B_5}{2z}.
\]

We will use this expression for the calculations, but first, let us confirm that the expression of the harmonic point above agrees with the definition. Given
\[
C = \frac{A + B}{2} + \frac{z(A - B)}{2}, \quad D = \frac{A + B}{2} + \frac{A - B}{2z},
\]
let us show that the cross-ratio is $-1$. The differences are

\[
A - C = \frac{(1 - z)(A - B)}{2},\quad
B - D = \frac{B - A}{2} \left( 1 + \frac{1}{z} \right),
\]

\[
A - D = \frac{(1 - \frac{1}{z})(A - B)}{2}, \quad
B - C = \frac{(1 + z)(B - A)}{2}.
\]

The cross-ratio becomes:
\[
(A, B; C, D) = \frac{\frac{(1 - z)\left(1 + \frac{1}{z}\right)(A - B)(B - A)}{4}}{\frac{\left(1 - \frac{1}{z}\right)(1 + z)(A - B)(B - A)}{4}}=-1.
\]

We now show that the points $A_4, B_3, C_1, C_2$ determine a harmonic point of the fifth triangle $T_5$ by proving

\[
D_5=\frac{A_4 + B_3}{2} + \frac{A_4 - B_3}{2z} + C_1 - C_2.
\]

Using $C_3=C_4$, we obtain

\begin{align*}
    \frac{A_3 + B_3}{2} + \frac{z(A_3 - B_3)}{2} &= \frac{A_4 + B_4}{2} + \frac{z(A_4 - B_4)}{2} \\
    \Longleftrightarrow A_3 \left( \frac{1+z}{2} \right) + B_3 \left( \frac{1-z}{2} \right) &= A_4 \left( \frac{1+z}{2} \right) + B_4 \left( \frac{1-z}{2} \right) \\
    \Longleftrightarrow A_3 \left( \frac{1+z}{2} \right) - B_4 \left( \frac{1-z}{2} \right) &= A_4 \left( \frac{1+z}{2} \right) - B_3 \left( \frac{1-z}{2} \right) \\
    \Longleftrightarrow A_3 \left( \frac{\frac{1}{z} + 1}{2} \right) - B_4 \left( \frac{\frac{1}{z} - 1}{2} \right) &= A_4 \left( \frac{\frac{1}{z} + 1}{2} \right) - B_3 \left( \frac{\frac{1}{z} - 1}{2} \right) \\
    \Longleftrightarrow \frac{A_3 + B_4}{2} + \frac{A_3 - B_4}{2z} &= \frac{A_4 + B_3}{2} + \frac{A_4 - B_3}{2z}.
\end{align*}

By considering the difference $C_1-C_2$, the other equation is
\begin{align*}
C_1 - C_2 &= \frac{A_1 + B_1}{2} + \frac{z(A_1 - B_1)}{2} - \left( \frac{A_2 + B_2}{2} + \frac{z(A_2 - B_2)}{2} \right) \\
          &= \frac{A_1}{2} - \frac{A_2}{2} + \frac{A_1}{2z} - \frac{A_2}{2z} = \frac{A_5}{2} - \frac{A_3}{2} + \frac{A_5}{2z} - \frac{A_3}{2z}.
\end{align*}

Putting the two calculations together, we obtained the required equality
\begin{align*}
\frac{A_4 + B_3}{2} + \frac{A_4 - B_3}{2z} + C_1 - C_2 &=\frac{A_3 + B_4}{2} + \frac{A_3 - B_4}{2z}+\frac{A_5}{2} - \frac{A_3}{2} + \frac{A_5}{2z} - \frac{A_3}{2z} \\
&=\frac{A_5+B_4}{2}+\frac{A_5-B_4}{2z}\\
&=\frac{A_5+B_5}{2}+\frac{A_5-B_5}{2z}=D_5. \qquad \qquad \qed
\end{align*}

\subsection{A geometric construction}

We now give an arrangement of $n^{1.726...}$ isosceles right triangles on $n$ points without a pair sharing their harmonic point (points $H, J, L$ in Figure \ref{fig:harmonic}). Our construction is based on Ruzsa's trick ``{\em much-more-differences-than-sums}'' \cite{Ru} (see~\cite{KT} for another application of this method). 

The bases of the triangles are spanned between two point sets, $A$ and $B$ along the axes. The following sums of complex numbers define the elements of the sets ($s\neq 1$ is a constant we will specify later; in fact, for concreteness, we will take $s=2$ though we leave the variable $s$ in the proof for clarity  of presentation):
\[
A=\left\{\sum_{k=1}^{3m} a_k13^k : a_k\in\{1,s\}, |\{k:a_k=1\}|=2m\right\},
\]
\[
B=\left\{\sum_{k=1}^{3m} b_k13^k : b_k\in\{i,is\}, |\{k:b_k=i\}|=m\right\}.
\]

Note that elements of $A$ and $B$ are determined uniquely by the coefficients $a_k, b_k$. 
With this definition, $|A|=|B|=\binom{3m}{m}$.

In our construction, two points, $\alpha=\sum_{k=1}^{3m} a_k13^k\in A$ and $\beta=\sum_{k=1}^{3m} b_k13^k\in B$ form the base of a triangle  if 
$$(a_k, b_k) \ne (s,i) \,\,\hbox{for all\,\,} k\in [3m].$$
Each $\alpha\in A$  forms a base with $\binom{2m}{m}$ distinct $\beta\in B$. Indeed, there are $2m$ coordinates where $a_k=1$ and $m$ coordinates where
$a_k=s$. In these latter $m$ coordinates $b_k=si$, so in the former $2m$ coordinates, $b_k=i$ for exactly $m$ of them.

The third point of the triangle, denoted $\gamma$, is uniquely determined by $\alpha$ and $\beta$ as 
\[
\gamma=\frac{\alpha+\beta}{2}+i\frac{\beta-\alpha}{2}=\sum_{k=1}^{3m}\frac{a_k(1-i) + b_k(1+i)}{2}13^k= \sum_{k=1}^{3m}g_k13^k.
\]
The angle at $\gamma$ is the right angle of triangle $\alpha\beta\gamma$, and $\gamma$ is below the base $\alpha\beta$.
%(See Figure~\ref{JSABpic}).
\iffalse
\begin{figure} [hbt!]
\hspace{50pt}\includegraphics[width=0.8\textwidth]{JSABpic.pdf}
  \caption{An $\alpha\beta\gamma$ triangle}\label{JSABpic}
 \end{figure}

\medskip

\begin{figure}[hbt]
\centering
\begin{tikzpicture}[scale=1.3]

  % Shade the area between y = -x and y = x
  \fill[green!30] (-3.2,3.2) -- (0,0) -- (3.2,3.2)  -- cycle;
   % Shade the area between y = -x and y = x
  \fill[blue!30] (3.2,3.2) -- (0,0) -- (3.2,-3.2)  -- cycle;

    % Draw x-axis
  \draw[->] (-3,0) -- (3,0) node[right] {};
  
  % Draw y-axis
  \draw[->] (0,-3) -- (0,3) node[above] {};
  
  % Draw the lines y = -x and y = x
  \draw[thick] (-3.2,3.2) -- (3.2,-3.2);
  \draw[thick] (-3,-3) -- (3.2,3.2);
  
  % Draw the point α at (-1, 2)
  \fill (-1,2) circle (2pt);
  \node[above] at (-1,2) {\(\alpha\)};
  
  % Draw the point β at (2, 1)
  \fill (2,1) circle (2pt);
  \node[above right] at (2,1) {\(\beta\)};
  
  % Draw the point γ at (2, 2)
 
  \fill (1,3) circle (2pt);
  \node[right] at (1,3) {\(\gamma\)}; 
  
  \draw[thick] (-1, 2) -- (1,3) -- (2,1) -- cycle; 
  
  % Draw the midpoint (α + β)/2 at (1/2, 3/2)
  \fill (0.5, 1.5) circle (2pt);
  \node[below] at (0.5, 1.5) {\(\frac{\alpha + \beta}{2}\)};
  
\draw[thick] (0.5, 1.5) -- (1,3);

\node at (-1.8, 2.5) {{\Large $A$}};
 \node at (1.8, -0.5) {{\Large $B$}};
  % Label the origin
%  \node[below left] at (0,0) {\(O\)};
\end{tikzpicture}

  \caption{An $\alpha\beta\gamma$ triangle}\label{JSABpic}
\end{figure}
\fi
The set of these  $\gamma$'s is denoted by $C$. The possible values of the $g_k$'s are $0$ and $\frac{(1-s)(1-i)}{2}$, as indicated in the tableau below. Moreover, exactly $2m$ values are $0$.

\[
\begin{array}{c|c|c}
a_k \backslash b_k & i & is  \\ \hline
1 & 0 & \frac{(1-s)(1-i)}{2} \\ 
s & \text{nil} & 0 \\ 

\end{array}
\]

With these definitions we have $|A|=|B|=|C|=\binom{3m}{m}$. We noted earlier that any $\alpha\in A$ is the vertex of $\binom{2m}{m}$ triangles, so the total number of triangles is 
\[
\binom{3m}{m}\binom{2m}{m}.
\]

It remains to prove that the harmonic points of those selected triangles are distinct. Given triangle $T=\alpha\beta\gamma$,  write $\delta_{\alpha}$ for the harmonic point of $T$ that lies on the opposite side of segment $\beta\gamma$ as $\alpha$,   
write $\delta_{\beta}$ for the harmonic point of $T$ that lies on the opposite side of segment $\alpha\gamma$ as $\beta$,
 and write $\delta_{\gamma}$ for the harmonic point of $T$ that lies on the opposite side of segment $\alpha\beta$ as $\gamma$. %(See Figure~\ref{JSHarmonicpts}).

The cross-ratio conditions for these points are the following:
$$(\alpha, \beta; \gamma, \delta_{\gamma})=-1 \qquad
(\gamma, \alpha; \beta, \delta_{\beta})=-1\qquad 
(\beta, \gamma; \alpha, \delta_{\alpha})=-1.
$$
Let us first analyze the case $\delta= \delta_{\gamma}$. In this case, $(\alpha, \beta; \gamma, \delta)=-1$  yields
\begin{align*}
\delta&=\frac{2\alpha \beta - \alpha\gamma - \beta \gamma}{\alpha+\beta - 2\ \gamma} \\
&=\frac{\alpha+\beta}{2}+\frac{\beta-\alpha}{2i}\\
&= \frac{\alpha+\beta}{2}-i\frac{\beta-\alpha}{2} \\
&=\sum_{k=1}^{3m}\frac{a_k(1+i) + b_k(1-i)}{2}13^k \\
&= \sum_{k=1}^{3m}d_k13^k.
\end{align*}
Note that we could immediately have obtained the third display $\delta=(\alpha+\beta)/2 - i(\beta-\alpha)/2$ by observing that $\alpha, \beta, \gamma, \delta$ form the corners of a square with diagonal $\alpha\beta$ so we obtain $\delta$ from the midpoint of the segment $\alpha\beta$ by moving in the direction opposite to that of $\gamma$.
The possible values of the $d_k$'s are $1+i$, $(1+s)(1+i)/2$, $s+is$, as indicated in the tableau below.

\[
\begin{array}{c|c|c}
a_k \setminus b_k & i & is  \\ \hline
1 & 1+i & \frac{(1+s)(1+i)}{2} \\ \hline
s &  \text{nil} & s+is \\ 

\end{array}
\]

For any 
$$\delta \in \left\{ \sum_{k=1}^{3m} d_k 13^k :d_k \in \left\{1+i, \frac{(1+s)(1+i)}{2}, s+is\right\} \right\},$$ 
there is a triangle $\alpha\beta\gamma$ with harmonic point $\delta$. As $s\neq 1$, from the digits of a harmonic point, we can uniquely recover the $\alpha,\beta$ base points of the triangles so no two triangles share their harmonic points. All of these harmonic points $\delta_{\gamma}$ are in the positive quadrant, and the others are outside this quadrant, so they do not overlap. To see that, note that the circumcircle of triangle $\alpha\beta\gamma$ goes through the origin and the points $\delta_\alpha, \delta_{\gamma}$ lie in the arc of this circle between $\beta\gamma$ and
between $\alpha\gamma$. Both these arcs are outside the first quadrant (see Figure \ref{fig:abc}).

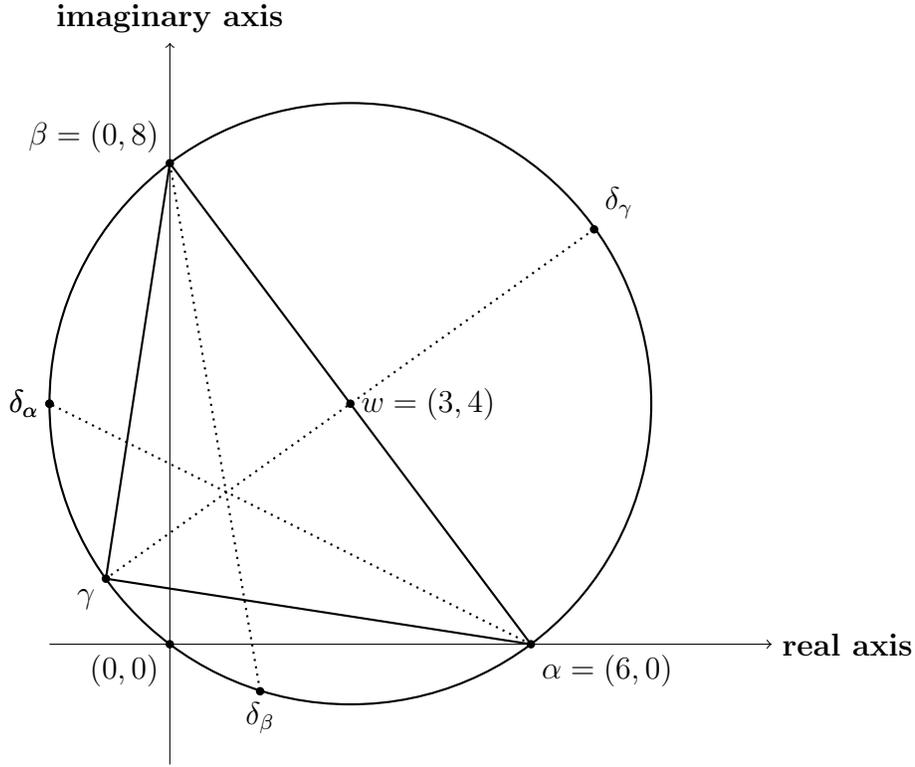
\begin{figure}
   \centering

\begin{tikzpicture}[scale=0.8]
    % Draw x and y axes
    \draw[->] (-2,0) -- (10,0) node[right] {\textbf{real axis}};
    \draw[->] (0,-2) -- (0,10) node[above] {\textbf{imaginary axis}};

    % Draw the circle with center (3,4) and radius 5
    \draw[thick] (3,4) circle (5);

    % Draw a line between alpha and beta
    \draw[thick] (6,0) -- (0,8);

 % Draw lines alpha-gamma, beta-gamma, and gamma-center
    \draw[thick] (6,0) -- (-17/16,1.09);  % Alpha to Gamma
    \draw[thick] (0,8) -- (-17/16,1.09);  % Beta to Gamma
    \draw[thick, dotted] (7.05, 6.9) -- (-17/16,1.09);  % Gamma to deltagamma
\draw[thick, dotted] (6, 0) -- (-2,4);  % alpha to deltaalpha
\draw[thick, dotted] (0, 8) -- (1.5,-0.782);  % beta to deltabeta
    % Mark and label the center
    \fill (3,4) circle (2pt) node[right] {$w=(3,4)$};

    % Mark and label intersection points
    \fill (0,0) circle (2pt) node[below left] {$(0,0)$};
    \fill (6,0) circle (2pt) node[below right] {$\alpha=(6,0)$}; % Label (6,0) as alpha
    \fill (0,8) circle (2pt) node[above left] {$\beta=(0,8)$}; % Label (0,8) as beta
\fill (-17/16,1.09) circle (2pt) node[below left] {$\gamma$}; % Gamma
\fill (7.05, 6.9) circle (2pt) node[above right] {$\delta_{\gamma}$}; % Label (7.1, 6.9) as delta_gamma
\fill (-2, 4) circle (2pt) node[left] {$\delta_{\alpha}$}; % Label (-2, 4) as delta_alpha
\fill (-2, 4) circle (2pt) node[left] {$\delta_{\alpha}$}; % Label (-2, 4) as delta_alpha
\fill (1.5, -0.782) circle (2pt) node[below] {$\delta_{\beta}$}; % Label (-2, 4) as delta_beta

\end{tikzpicture}
  \caption{Example of triangle $\alpha\beta\gamma$ and its harmonic points}  \label{fig:abc}
\end{figure}

\iffalse
\begin{figure}
\begin{center}
    \includegraphics[angle=270,scale=.5]{Harmonic3.pdf}
  \caption{Example of triangle $\alpha\beta\gamma$}  \label{fig:abc}
\end{center}
\end{figure}
\fi

There are two more harmonic points to consider for each triangle.

Recall $\ga=(\al+\be)/2 + i (\be-\al)/2$, and $(\ga,\al;\be,\de_\beta)=-1$. Consequently, 
\begin{align*}
\de_{\beta} &=\frac{2\al\ga - \be\ga -\be\al}{\al+\ga-2\be} \\ 
&=\frac{ 2\al^2-\al\be-\be^2+i(3\al\be-2\al^2-\be^2)}{(\al-\be)(3-i)} \\
&=\frac{(2-2i)\al+(1+i)\be}{3-i}\\\
&=\frac{4-2i}{5}\al +\frac{1+2i}{5}\be. 
\end{align*}

For the sake of simplicity, we will count the number of different $\delta'=5\delta_\beta$ values. As before, we check the results digit-wise of 
$\delta'=(4-2i)\alpha +(1+2i)\beta$ 
for the possible $\alpha,\beta$ combinations. The results are summarized in the tableau below.

\[
\begin{array}{c|c|c}
a_k \setminus b_k & i & is  \\ \hline
1 & 2-i  & (2-s)(2-i) \\ \hline
s  & \text{nil} &  s(2-i)\\ 
\end{array}
\]

As before, by the digits of $\delta'$ we can recover the base of the triangle uniquely.

The harmonic point $\delta_{\alpha}$ is obtained by reflecting $\delta_{\beta}$ about the line segment $w\gamma$, where $w=(\alpha+\beta)/2$ is the midpoint of the base, and $\gamma=(\alpha+\beta)/2 + i (\beta-\alpha)/2$ is the third point of the triangle (see Figure \ref{fig:abc}).
An easy calculation now yields
$$\delta_{\alpha} = \frac{1-2i}{5} \alpha + \frac{4+2i}{5}\beta.$$

\iffalse
\begin{table}[]
    \centering
    \begin{tabular}{c|c}
         &  \\
         & 
    \end{tabular}
    \caption{Caption}
    \label{tab:my_label}
\end{table}
\fi

We set $\delta''=5\delta_\alpha=(1-2i)\alpha+(4+2i)\beta$ for the remaining harmonic point. The possible digit-wise entries of $\delta''$ are

\[
\begin{array}{c|c|c}
a_k \setminus b_k & i & is  \\ \hline
1 & -1+2i  & (1-2s)(1-2i) \\ \hline
s  & \text{nil} &  s(2i-1)\\ 
\end{array}
\]

Now we want to choose $s$ such that the sets of points of the two harmonic points are disjoint. It can be achieved for example by setting $s={2}$, when the digits of $\delta'$ are $2-i, 0, 4-2i$ and of $\delta''$ are $-1+2i,-3+6i, 4i-2$.

In the construction there are $n=3\binom{3m}{m}$ points and $\binom{3m}{m}\binom{2m}{m}$ triangles with disjoint harmonic points. Define $x$ as
\[
\left(3\binom{3m}{m}\right)^{\,x} \;=\;\binom{3m}{m} \binom{2m}{m}.\]
Taking logarithms and  letting $m\to \infty$ leads to 
\[
x = 
1+ \frac{\log_2 {2m \choose m}}{\log_2{3m \choose m}}+o(1)=
1 + \frac{2m+o(m)}{3H(1/3)m+o(m)}\] 
where $H(p)=-p\log_2p - (1-p)\log_2(1-p)$ is the binary entropy function. For large $m$, we obtain
$$x \sim 1+\frac{2}{3H(1/3)} = 1+ \frac{2}{3\log_2 3 - 2} \approx 1.726.$$

\section{Proof of Theorem \ref{thm:densetriangles}}

Recall that we are to prove the following:
 For every $c>0$ and $\varepsilon>0$, the following holds for large enough $n$.
Let $T$ be a triangle and $S$ be a set of $n$ points in the plane such that $S$ contains  $cn^2$ triangles similar to $T$. Then, there is a quadrangle $Q$ and a set of at most $(2c/\varepsilon^6)n$ points, denoted by $U$, such that $U$ contains at least $(c-\varepsilon)n^2$ quadrangles similar to $Q$ and $S\subset U$.

We are going to use the counting methods from the proof of Theorem \ref{thm:geom}. 
The proof follows a simple algorithm.  
Select one of the harmonic points of $T$. These four points will give $Q$. For any triangle similar to $T$, we will only consider the harmonic point, which gives a quadrangle similar to $Q$.
The set of the selected harmonic points is denoted by $H$. Set $\delta=\varepsilon^6$.

\begin{enumerate}
    \item Let us begin with $U=S.$
    \item Select a point $h\in H$ which is not in $U$ and the harmonic point of at least $\delta n$ triangles.
    \item If there is no such point, then stop.
    \item Set $U=U\cup h$ and repeat from step 2.
\end{enumerate}

The proof of Theorem~\ref{thm:geom} shows that there exist $m^6/n^{11}$ triangles sharing the same harmonic point, hence for any $c>0$, as $n$ is sufficiently large and the number of triangles on $n$ points similar to  $T$ is $m=cn^2> 60 n^{11/6}$,  there are at least $m^6/n^{11}= c^6n$ triangles sharing a given harmonic point.  In every step we selected at least $\delta n$ triangles and no triangle was selected multiple times. Hence the number of iterations in the algorithm is at most $cn^2/(\delta n)$ and $|U|\leq |S|+ cn^2/(\delta n) =(c/\delta+1)n \le  (2c/\varepsilon^6)n$. Also, this selection of $\delta$ guarantees that all but at most $\varepsilon n^2$ triangles have their harmonic points in $U$. For if there are more than $\varepsilon n^2 > 60n^{11/6}$ triangles with harmonic point not in $U$, then by Theorem~\ref{thm:geom}, there are at least $(\varepsilon n)^{6}/n^{11}=\varepsilon^6 n = \delta n$ triangles that share a common harmonic point and the algorithm would not have terminated.

\end{document}